\documentclass[pdflatex,sn-mathphys-num]{sn-jnl}

\usepackage{graphicx}
\usepackage{multirow}
\usepackage{amsmath,amssymb,amsfonts}
\usepackage{amsthm}
\usepackage{mathrsfs}
\usepackage[title]{appendix}
\usepackage{xcolor}
\usepackage{textcomp}
\usepackage{manyfoot}
\usepackage{booktabs}
\usepackage{algorithm}
\usepackage{algorithmicx}
\usepackage{algpseudocode}
\usepackage{listings}

\usepackage{hyperref}
\newcommand{\orcidlink}[1]{\href{https://orcid.org/#1}{\includegraphics[width=10pt]{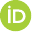}}}

\theoremstyle{thmstyleone}
\newtheorem{theorem}{Theorem}
\newtheorem{proposition}[theorem]{Proposition}

\newtheorem{cond}[theorem]{Condition}

\theoremstyle{thmstyletwo}

\newtheorem{remark}{Remark}

\theoremstyle{thmstylethree}

\begin{document}

\title[Host--Parasitoid Dynamics and Biological Control of the Sugarcane Borer]{Host--Parasitoid Dynamics and Biological Control of the Sugarcane Borer}

\author[1]{\fnm{Fredson} \sur{Aguiar}\,\orcidlink{0000-0002-2074-4504}}
\author*[1]{\fnm{M. Soledad} \sur{Aronna}\,\orcidlink{0000-0002-0640-4722}}\email{soledad.aronna@fgv.br}
\author[2]{\fnm{Roberto} \sur{Guglielmi}\,\orcidlink{0000-0002-4443-2695}}\author[1]{\fnm{Beatriz} \sur{Laiate}\,\orcidlink{0000-0002-7312-4843}}\author[3]{\fnm{Alexandre}\sur{Molter}\,\orcidlink{0000-0001-8562-6376}}

\affil*[1]{\orgdiv{School of Applied Mathematics}, \orgname{Getulio Vargas Foundation}, \orgaddress{
\city{Rio de Janeiro},
\country{Brazil}
}}
\affil[2]{\orgdiv{Department of Applied Mathematics}, \orgname{University of Waterloo}, \orgaddress{
\city{Waterloo},
\country{Canada}
}}
\affil[3]{\orgdiv{Department of Mathematics and Statistics}, \orgname{Federal University of Pelotas}, \orgaddress{
\city{Pelotas},
\country{Brazil}
}}

\abstract{
We investigate biological pest control strategies for the sugarcane borer \textit{Diatraea saccharalis} through the combined action of two parasitoid species: the egg parasitoid \textit{Trichogramma galloi} and the larval parasitoid \textit{Cotesia flavipes}. We describe the population dynamics with a six-dimensional host–parasitoid model in which host–parasitoid interactions are represented through a Holling Type II functional response, extending previous models by coupling egg and larval stage dynamics and incorporating parasitism saturation observed in laboratory experiments. We characterize the equilibrium structure of the model and analyze the local stability of the extinction equilibrium. Bifurcation analysis reveals that, for a wide range of Holling parameters, the pest population exceeds the economic damage threshold, motivating the design of active control strategies. We formulate and compare three biological control approaches: open-loop optimal control, State-Dependent Riccati Equation (SDRE) feedback control, and impulsive feedback control based on Lyapunov arguments. We perform numerical simulations to show that all three strategies successfully keep the pest population below the economic damage threshold. The impulsive strategy, in particular, achieves effective suppression with substantially fewer parasitoid releases than the continuous approaches, making it the most practically viable option for field implementation.}



\maketitle

\section{Introduction}

Economic interest in sugarcane production has grown significantly in recent decades, driven primarily by its key role in sugar and ethanol production. This expansion has been accompanied by increasing pest pressure on crops, with the sugarcane borer {\em Diatraea saccharalis} representing the most economically significant threat. Damage begins as early as the third month after planting, during the tillering stage, and is further favored by high rainfall and elevated temperatures in spring and summer. Newborn larvae scrape the leaves and bore into the softer tissue of the stalk, creating internal galleries that cause direct damage — including sprout death, leaf desiccation, weight loss, and atrophy — as well as indirect damage through yeast contamination, which leads to stem rot, reduced juice purity, and significant industrial yield loss in sugar and ethanol production
\cite{PARRA, Parra2021,Vargas}.

The use of insecticides against the sugarcane borer in the larval stage is unfeasible and ineffective because the larvae are within the stalk of the plant. A viable and efficient alternative is based on biological control that aims to reduce the pest population using sugarcane borer parasitoids in different stages of the borer development. Two main parasitoids have been successfully deployed in Brazil and other sugarcane-producing countries for the biological control of the sugarcane borer~\cite{PARRA,PARRA2}: the wasp {\it Trichogramma galloi} ({\it T. galloi}) for the egg stage of the borer, and  the wasp {\it Cotesia flavipes} ({\it C. flavipes}) for its larval stage. These two parasitoids are mass-reared in laboratories or biofactories and made available to be released in sugarcane plantations, where they parasitize the eggs and larvae of the sugarcane borer by laying inside them their own eggs, which leads to a reduction of the overall pest population.

This work proposes a novel mathematical model for the dynamic interactions between the sugarcane borer and its two parasitoids, the egg-stage parasitoid {\it T. galloi} and the larval-stage parasitoid {\it C. flavipes}. This interaction is described by a compartmental host-parasitoid model, where the parasitism is expressed in terms of a Holling Type II functional response~\cite{Holling,Royama,Rogers,Pritchard}. 
Moreover, we develop a systematic analysis of the biological control of the sugarcane borer by means of controlled releases of its two parasitoids.

Biological control strategies for agricultural pests have increasingly
incorporated mathematical modelling to support decision-making and optimise
intervention protocols.
For the sugarcane borer \textit{Diatraea saccharalis}, control programmes
typically rely on the sequential release of parasitoids targeting different
developmental stages of the host: the egg parasitoid \textit{T.~galloi} and
the larval parasitoid \textit{C.~flavipes}, which attack the egg and larval
stages, respectively~\cite{MOLNAR}.
The theoretical foundation for this two-stage, two-parasitoid structure was
laid by~\cite{Briggs}, and subsequently developed into a stage-structured
host--parasitoid model by~\cite{Molter}, who focused specifically on the
sugarcane borer system.
Building on that framework,~\cite{MAR2} introduced an age-structured model
that incorporates seasonal perturbations to better reflect environmental
variability in the field.
The present article extends the dynamics of~\cite{MAR2} by replacing the
mass-action parasitism term with a Holling Type~II functional response,
which accounts for parasitism saturation at high host densities as observed
in laboratory experiments~\cite{Santos, Santos1}.

In practice, parasitoids are released in discrete quantities at specific
times --- a strategy formalised as impulsive control in Control
Theory~\cite{BARC, Marat2012, TANG} --- and periodic release schedules,
such as weekly applications, are the most operationally feasible option
under field conditions~\cite{Botelho, Gomes, Nava}.
Continuous release frameworks, by contrast, assume that inputs can be
adjusted at every instant. In contrast, continuous Optimal Control yields theoretically minimal-cost
strategies and provides useful lower bounds against which discrete protocols
can be benchmarked~\cite{MOLNAR, Molter}, while the State-Dependent Riccati
Equation (SDRE) method~\cite{MRACEKKK, MRACEK} produces dynamically adjusted
feedback rates without requiring the solution of a global optimization problem.
Both continuous approaches face practical limitations arising from
implementation costs and logistical constraints, motivating hybrid strategies
in which releases are triggered impulsively whenever pest densities exceed the
economic damage threshold, combining operational feasibility with
optimisation-informed timing~\cite{Molter}.
In this study we formulate and compare all three paradigms --- optimal
control, SDRE feedback, and threshold-triggered impulsive release --- within
a unified host--parasitoid model, assessing their effectiveness and
suitability for field implementation.

Summarizing, the main contributions of this article are: (i) a novel six-dimensional host–parasitoid model incorporating Holling Type II functional responses for both parasitoid species and a larva-dependent egg recruitment term; (ii) a rigorous characterisation of the five equilibrium points and an explicit stability criterion for the trivial (pest-free) equilibrium; and (iii) a comparative analysis of three biologically grounded control strategies—optimal control, SDRE feedback, and impulsive feedback—quantified by total parasitoid release and final pest density.

The paper is organised as follows.
Section~\ref{sec:model} introduces the six-dimensional host--parasitoid
model, establishes the existence and local stability of the equilibria,
and carries out a bifurcation analysis with respect to the Holling Type~II
handling-time parameters.
Section~\ref{SecControl} formulates the three biological control
strategies: an optimal control problem solved via Pontryagin's Maximum
Principle, a continuous feedback law derived from the State-Dependent
Riccati Equation (SDRE), and a threshold-triggered impulsive release law
supported by a Lyapunov stability argument.
Section~\ref{Sec:Discussion} presents a numerical comparison of the three
strategies in terms of pest suppression performance, parasitoid release
costs, and robustness to parameter uncertainty.
Section~\ref{SecConclusion} summarises the main findings, discusses their
practical implications for sugarcane borer management programmes, and
outlines directions for future research.

\section{The mathematical model}\label{sec:model}

We now introduce the host-parasitoid system describing the interaction between the sugarcane borer {\em D. saccharalis} and its two parasitoids for its egg- ({\it T. galloi}) and larval-stage ({\it C. ﬂavipes}), respectively, with a Holling Type II functional response \cite{Holling}.
The proposed model is given by:
\begin{equation}
    \begin{split}
    \frac{d x_1}{dt} & = r \sigma_1 x_4 \left(1 - \frac{x_1}{K}\right) - m_1 x_1 - n_1 x_1 - \frac{a_1 x_1x_3}{1 + a_1h_1x_1}, \\
    \frac{d x_2}{dt} &= \frac{a_1 x_1x_3}{1 + a_1h_1x_1} - m_2 x_2 - n_2 x_2, \\
    \frac{d x_3}{dt} &= \gamma_3 n_2 x_2 - m_3 x_3, \\
    \frac{d x_4}{dt} &= n_1 x_1     - m_4 x_4 - n_3 x_4  -  \frac{a_2x_4x_6}{1 + a_2h_2x_4} , \\
    \frac{d x_5}{dt} &= \frac{a_2 x_4x_6}{1 + a_2h_2x_4}  - m_5 x_5 - n_4 x_5, \\
    \frac{d x_6}{dt} &= \gamma_6 n_4 x_5 - m_6 x_6, \label{eq1}    
    \end{split}
\end{equation}
where the state variables and parameters are described in Tables~\ref{tab:populations} and~\ref{table:parameters}, respectively. As a standing assumption of the paper, we assume that all the system's parameters are positive. 
The recruiting term in~\eqref{eq1} is considered proportional to the size of the larvae population, in view of the absence of adult population in the model, and also proportional to $ 1 - \frac{x_1}{K},$ which represents a space/resource limitation in terms of the carrying capacity $K$ for the eggs. The closer $x_1$ is to $K$ the lower is the recruiting term. See further explanation in, {\em e.g.},~\cite{Antunes}.

\begin{table}[!ht] 
    \caption{Description of the state variables of the model}\label{tab:populations}
    \begin{tabular}{c c}
        \toprule
        Variable & Description \\
        \midrule
         $x_1$ & Un-parasitized egg population density of the sugarcane borer \\
         $x_2$ & Parasitized egg population density of the sugarcane borer \\
         $x_3$ & Population density of the adult egg parasitoid {\em T. galloi } \\
         $x_4$ & Un-parasitized larvae population density of the sugarcane borer \\
         $x_5$ & Parasitized larvae population density of the sugarcane borer \\
         $x_6$ & Population density of the adult larvae parasitoid {\em C. flavipes} \\
         \botrule
    \end{tabular}  
\end{table}

\begin{figure}[ht]
    \centering
    \includegraphics[width=1\linewidth]{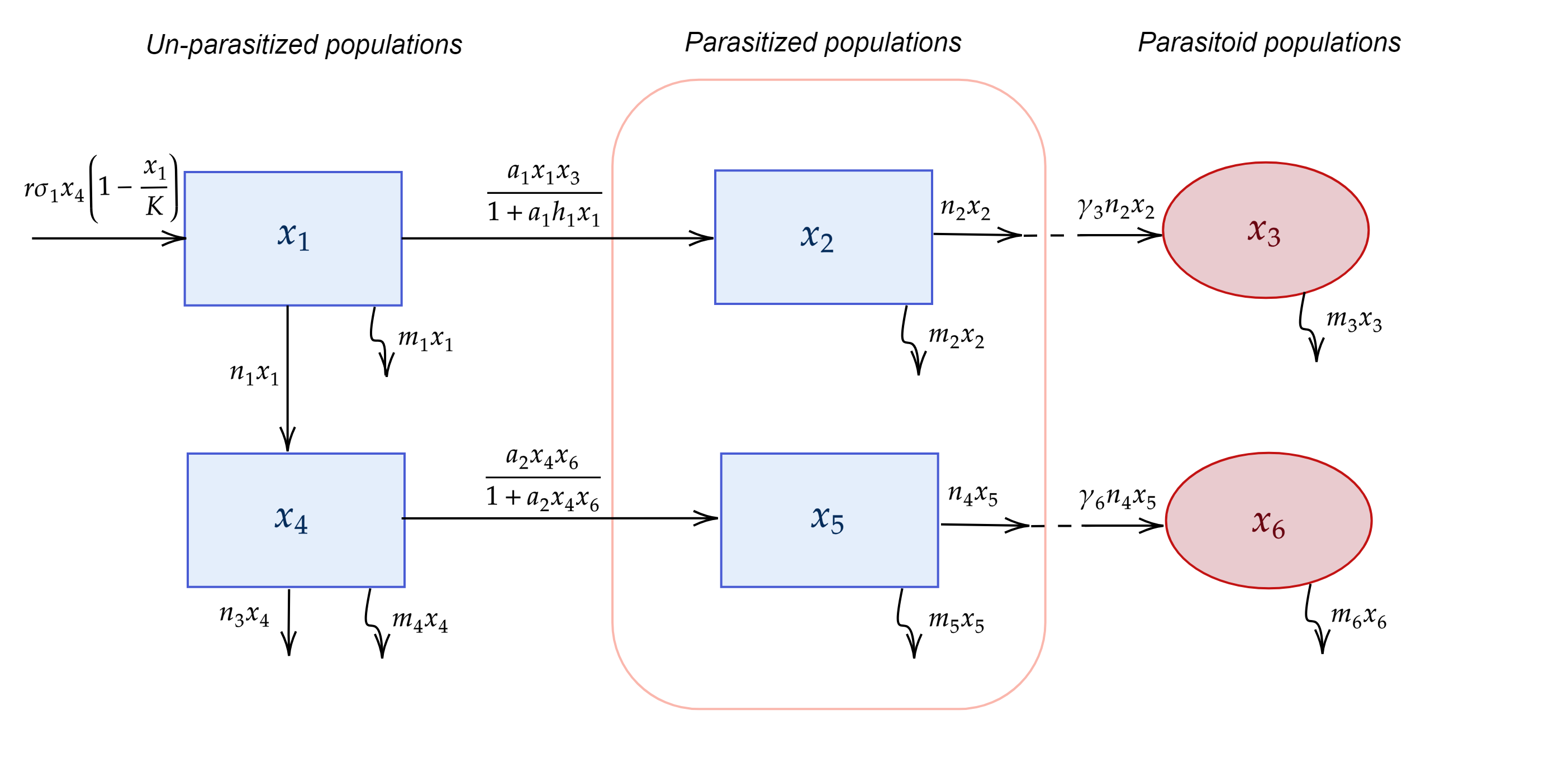}
    \caption{
    Diagram of the host-parasitoid dynamics~\eqref{eq1}. The full arrows represent the contributions between interacting populations, and the curved arrows represent the mortality rates. From left to right, the figure depicts the un-parasitized, the parasitized, and the parasitoids populations, respectively.}
    \label{fig:enter-label}
\end{figure}

\begin{table}[ht]
\centering
\caption{Description of the parameters of the model}\label{table:parameters}
\begin{tabular}{cc}
\toprule
Parameter & Description                  \\ \midrule
$r$                & Intrinsic oviposition rate of female sugarcane borer \\
$K$                & Carrying capacity of egg population \\
$\sigma_1$         &  Proportion of adult female sugarcane borers among larvae \\
$a_1$, $a_2$       &  Effective encounter rate between parasitoid and susceptible host  \\
$h_1$, $h_2$       &  Average handling time per parasitized host \\
$m_1$              & Mortality rate of the un-parasitized egg population \\
$m_2$              & Mortality rate of the parasitized egg population \\
$m_3$              & Mortality rate of the {\em T. galloi} population \\
$m_4$              & Mortality rate of the un-parasitized larvae population \\
$m_5$              & Mortality rate of the parasitized larvae population \\
$m_6$              & Mortality rate of the {\em C. flavipes} population \\
$n_1$              & \begin{tabular}[c]{@{}c@{}}Fraction of the sugarcane borer larvae population which emerges \\ from egg per unit time\end{tabular} \\
$n_2$ & \begin{tabular}[c]{@{}c@{}}Fraction of the parasitized egg population from which larvae \\ parasitoids emerge per unit time\end{tabular} \\
$n_3$              & \begin{tabular}[c]{@{}c@{}}Per-capita loss rate of un-parasitized sugarcane borer larvae \\ due to pupation (larvae exit this stage; pupae are not tracked)\end{tabular} \\
$n_4$ & \begin{tabular}[c]{@{}c@{}}Fraction of the parasitized sugarcane borer larvae from which larvae \\ parasitoid emerge per unit time\end{tabular} \\
$\gamma_3$         &   Number of female adult parasitoids that emerge from a parasitized egg unit          \\
$\gamma_6$         &  Number of female adult parasitoids that emerge from a parasitized larvae unit \\
\botrule
\end{tabular}
\end{table}

In comparison to \cite{Molter, MAR2}, the formulation in \eqref{eq1} differs in several aspects: (i) it accounts for egg–larva interaction, which requires a modified recruitment term for the egg population;
(ii) more importantly, it models both host-parasitoid interactions, between the egg parasitoid {\it T. galloi} and the egg population, as well as between the larval parasitoid {\it C. flavipes} and the larval population, using Holling Type II functional responses, instead of the Holling Type I (mass-action) used in \cite{Molter, MAR2}. 
Moreover, the egg recruitment now depends on larval density, whereas in \cite{Molter, MAR2} it depends only on eggs, which decouples the first three equations from the remaining ones. This structural change leads to substantially different dynamics, as illustrated in the following section.

\subsection{Properties of the model}

Here we establish some fundamental invariant properties of system~\eqref{eq1}, ensuring that the solutions remain biologically meaningful for all admissible initial conditions.
Let us set
 \[
 \mathcal{C} := \left\{ {x(t)} \in \mathbb{R}_+^6: x_1(t) \leq K\right\}.
 \]
We obtain the following result, inspired by the analysis in~\cite{haddad2010} on the positivity of nonlinear systems.
\begin{theorem}[Invariance of $\mathcal{C}$] \label{the:metzlerpos} For any initial nonnegative condition ${x}_0\in \mathbb{R}_+^6$, there exists a unique solution $x$ to the initial value problem associated with~\eqref{eq1}. Moreover, the solution is defined in $[0,+\infty)$ and it remains non-negative for all $t\ge 0$. 
Furthermore, if the initial condition lies in  $\mathcal{C}$,  then the corresponding trajectory $x(t)$ remain in $ \mathcal{C},$ for all $t \in [0,+\infty).$
\end{theorem}

\begin{proof}
    Assume that the trajectory $x_1$ hits zero at some time $t_1 > 0$, while the other trajectories remain non-negative. Evaluating the first equation of \eqref{eq1} at $t = t_1$, we get
\begin{equation}\label{eqproof1}
\frac{d x_1(t_1)}{dt} = r \sigma_1 x_4(t_1) \geq 0.
\end{equation}
 Thus, $\frac{d x_1(t_1)}{dt}$ is non-negative and, consequently, $x_1$ never assumes negative values.
An analogous argument applies to the other variables in system \eqref{eq1}, ensuring that all of them remain non-negative for all times.

Now, it remains to verify that $x_1(t) \leq K$, for all $t \geq 0$, provided that the initial condition lies in $\mathcal{C}.$
Suppose, by contradiction, that there exists a time $t_1 > 0$ such that $x_1(t_1) > K$. Then, necessarily, there must be a time $t^*$,  with $0 \leq t^*<t_1$, such that $x_1(t^*) = K$
and $\frac{d x_1(t^*)}{dt} > 0$. Evaluating the first equation of \eqref{eq1} at $t = t^*$, we get:
\begin{equation}\label{eqproof2}
    \frac{d x_1(t^*)}{dt} = - m_1 x_1(t^*) - n_1 x_1(t^*) - \frac{a_1 x_1(t^*) x_3(t^*)}{1 + a_1 h_1 x_1(t^*)} \leq 0.
\end{equation}
This contradiction implies that $x_1(t) \leq K$ for all $t \geq 0$. 

We conclude that $x(t) \in \mathcal{C}$,  for all $t \geq 0$, if the initial condition verifies $x(0) \in \mathcal{C}$.
\end{proof}

\subsection{Equilibria}

The expressions for the equilibrium points of system~\eqref{eq1} are presented in the following statement. The corresponding proof is provided in Appendix~\ref{AppEqui}, while the involved hypotheses are collected in Condition~\ref{assump:equilibria} below.

\begin{proposition}[Equilibria]\label{prop:equilibria}
System~\eqref{eq1} admits the following five equilibrium points:
\begin{enumerate}
    \item The trivial equilibrium $
    E^0 = (0,0,0,0,0,0).$

    \item The parasitoid-free equilibrium $
    E^{Pf} = (x_1^{Pf},0,0,x_4^{Pf},0,0),$ 
    which is biologically feasible provided that Condition~\ref{assump:equilibria}.\ref{itma} holds.

    \item The equilibrium corresponding to the absence of adult egg parasitoids \emph{T. galloi} and the persistence of adult larvae parasitoids \emph{C. flavipes}, given by $
    E^{Cf} = (x_1^{Cf},0,0,x_4^{Cf},x_5^{Cf},x_6^{Cf}),$ 
    which is biologically feasible under Conditions~\ref{assump:equilibria}.\ref{itmb}--\ref{itmc}.

    \item The equilibrium corresponding to the absence of adult larvae parasitoids \emph{C. flavipes} and the persistence of adult egg parasitoids \emph{T. galloi}, given by $
    E^{Tg} = (x_1^{Tg},x_2^{Tg},x_3^{Tg},x_4^{Tg},0,0),$ 
    which is biologically feasible under Conditions~\ref{assump:equilibria}.\ref{itmd}--\ref{itme}.

    \item The coexistence equilibrium $E^{\rm co} = (x_1^{co},x_2^{co},x_3^{co},x_4^{co},x_5^{co},x_6^{co}),$ 
    which is biologically feasible under Conditions~\ref{assump:equilibria}.\ref{itmb}--\ref{itmd}, \ref{itmf}, and \ref{itmg}.
\end{enumerate}
\end{proposition}

\begin{cond}\label{assump:equilibria} We assume the following relations among the system's parameters:
   \begin{enumerate}
\setlength{\itemsep}{1.2ex}
    \item \label{itma} $\quad (m_1 + n_1)(m_4 + n_3) < r\sigma_1 n_1$.
    \item \label{itmb} $\quad \gamma_6 n_4 - (m_5 + n_4) h_2 m_6 >0$.
    \item \label{itmc} $\displaystyle\quad \frac{(m_5 + n_4)m_6}{a_2[\gamma_6 n_4 - (m_5 + n_4)h_2 m_6]} < K\,\frac{r n_1\sigma_1 -(m_1 + n_1)(m_4 + n_3)}{r\sigma_1 (m_4 + n_3)}$.
    \item \label{itmd} $\quad \gamma_3 n_2 - (m_2 + n_2) m_3 h_1 >0$.
    \item \label{itme} $\quad \displaystyle\frac{(m_2 + n_2)m_3}{a_1[\gamma_3 n_2 - (m_2 + n_2) m_3 h_1]} < K\left[1 - \frac{(m_1 + n_1)(m_4 + n_3)}{r\sigma_1 n_1}\right]$.
    \item \label{itmf}{\footnotesize
    $\quad \displaystyle
    \frac{(m_2 + n_2) m_3}
    {a_1 \big[ \gamma_3 n_2 - (m_2 + n_2) m_3 h_1 \big]}
    <
    K\,\frac{r\sigma_1 (m_5 + n_4) m_6}
    {r\sigma_1 (m_5 + n_4) m_6 + K(m_1 + n_1) + a_2 \big[\gamma_6 n_4 - (m_5 + n_4) h_2 m_6\big]}$.}
    \item \label{itmg} $\quad \displaystyle
    \frac{m_4 + n_3}{n_1}\,
    \frac{(m_5 + n_4) m_6}
    {a_2 \big[\gamma_6 n_4 - (m_5 + n_4) h_2 m_6\big]}
    <
    \frac{(m_2 + n_2) m_3}
    {a_1 \big[ \gamma_3 n_2 - (m_2 + n_2) m_3 h_1 \big]}$.
\end{enumerate}
\end{cond}

\subsection{Local stability analysis}

In this section, we investigate the local stability properties of the equilibrium point $E^0$, which is the only equilibrium point for which pest populations vanish. Recall that an equilibrium point $E^*$ is locally asymptotically stable (LAS) if and only if all eigenvalues of the Jacobian matrix evaluated at $E^*$ have negative real parts. This characterization allows us to derive necessary and sufficient conditions for the stability of $E^0$ in terms of the system parameters.

We obtain the following result, whose proof is presented in Appendix~\ref{appendixB}.

\begin{proposition}[Stability of the origin] \label{prop:staborigin}
The trivial equilibrium $E^0 = (0,0,0,0,0,0)$ of~\eqref{eq1} is stable if and only if
    \begin{equation} \label{eq:eigenvorigin3}
        (m_1+n_1)(m_4+n_3) >
        r n_1\sigma_1.
    \end{equation}
\end{proposition}

\begin{remark}
It is worth noting that, for the parameter values given in Table~\ref{tabmodel}, the origin is an unstable equilibrium of system~\eqref{eq1}. This aligns with empirical observations: the host-parasitoid populations 
coexist in nature, so one expects the trivial equilibrium to be unstable and possibly the coexistence equilibrium to be stable. As a matter of fact, Condition~\ref{assump:equilibria}.\ref{itma} is complementary to~\eqref{eq:eigenvorigin3}, and thus guarantees the instability of $E^0$.
\end{remark}

\begin{table}[ht]
\caption{Value of parameters of system (\ref{eq1})}\label{tabmodel}
\begin{tabular}{c c c c c}
\toprule
Parameter &
Value &  Reference \\
\midrule
$r$ &
0.19 & \cite{PARRA} \\
$K$ &
25000 & \cite{JOAO,Molter, MAR,Wiedenmann} \\
$m_1$ &
$10^{-3}$ & Assumed $m_1 <<1$ \\
$m_2$ &
0.03566 & \cite{PARRA} \\
$m_3$ &
1/4 & \cite{PARRA} \\
$m_4$ &
0.00257 & \cite{PARRA} \\
$m_5$ &
$2.73m_4$ & \cite{Parra2021} \\
$m_6$ &
1/5 & \cite{PARRA} \\
$n_1$ &
1/8 & \cite{JOAO,Molter, MAR} \\
$n_2$ &
1/9 & \cite{JOAO,Molter, MAR} \\
$n_3$ &
1/50 & \cite{JOAO,Molter, MAR} \\
$n_4$ &
1/16 & \cite{JOAO,Molter, MAR} \\
$\gamma_3$&
1 & \cite{PARRA} \\
$\gamma_6$ &
3 & \cite{JOAO, PARRA}\\
$\sigma_1$ &
0.5 & \cite{PARRA, Parra2021} \\
$a_1$ &
0.02981 & Computed from \cite{Lopes} \\
$a_2$ &
0.041 &  Computed from \cite{Wiedenmann} \\
$h_1$ &
7.02651 & Computed from \cite{Lopes}\\
$h_2$ &
0.012 & Computed from \cite{Wiedenmann} \\
\botrule
\end{tabular}
\end{table}

\subsection{Numerical simulations of the dynamical system}

In this section, we present some numerical simulations of system \eqref{eq1} to illustrate the system’s sensitivity to selected parameters' variations. In doing so, we highlight the occurrence of undesirable state trajectories, since the pest populations exceed an {\em economic damage threshold}, which motivates the need for biological control strategies, analyzed in the subsequent sections.
 The codes used in this article are available at a public repository\footnote{Access public repository: \url{https://github.com/fredsonaguiar/parasitoid-sugarcane-control}.}. 

We first consider a parasitoid-free scenario and illustrate in Figure \ref{fig:1} the state trajectories of the populations $x_1$ and $x_4$ starting from the initial conditions
\begin{equation} \label{eq:initial1}
    \begin{bmatrix}
    x_1(0), x_2(0), x_3(0), x_4(0), x_5(0), x_6(0)
\end{bmatrix} = \begin{bmatrix}
    800, 0, 0, 1000, 0, 0
\end{bmatrix}     
\end{equation}
and the parameters given in Table \ref{tabmodel}.
The emergence of the sugarcane borer in the field may lead to significant economic losses. For this reason, it is necessary to define an {\em economic damage threshold}, which represents the upper density of borer's larvae per hectare beyond which the crop is considered lost. In the present study, we fix the value of the economic damage threshold at $2500$ larvae/ha \cite{Lopes, Molter, Nava,MAR}. Since population parameters are not available directly under field conditions, this economic damage threshold was adopted from the literature for pest control with model parameters based on laboratory data.

\begin{figure}[ht]
    \centering
    \includegraphics[width=1\linewidth]{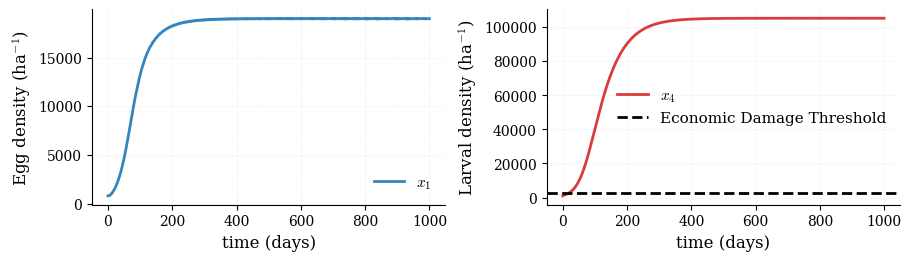}
    \caption{Simulation of the parasitoid-free model with initial condition~\eqref{eq:initial1}
and parameter values from Table~\ref{tabmodel}.
On the left: un-parasitized egg density $x_1$ (eggs~ha$^{-1}$).
On the right: un-parasitized larval density $x_4$ (larvae~ha$^{-1}$);
the dotted horizontal line marks the economic damage threshold of $2500$ larvae~ha$^{-1}$.
Both trajectories converge to the corresponding coordinate of the parasitoid-free equilibrium $E^{P_f}$,
with asymptotic values $x_1^{P_f}\approx 19013$ eggs~ha$^{-1}$
and $x_4^{P_f}\approx 105300$ larvae~ha$^{-1}$.}
\label{fig:1}
\end{figure}

In Figure \ref{fig:1}, we observe that the trajectories approach the parasitoid-free equilibrium $E^{P_f}$ which, in view of Proposition~\ref{prop:equilibria} and the values in Table~\ref{tabmodel}, is given by  $E^{P_f} \approx (19013,0,0,105300,0,0).$ We also see that, at this equilibrium, $x_4$ largely overpasses the economic damage threshold.

On the other hand, in the case of the presence of parasitoids in the field, we consider the initial conditions with non-zero values for the parasitoid populations
\begin{equation} \label{eq:initial2}
    \begin{bmatrix}
        x_1(0), x_2(0), x_3(0), x_4(0), x_5(0),
        x_6(0) 
    \end{bmatrix} = \begin{bmatrix}
        800, 0, 100, 1000, 0, 100 
    \end{bmatrix}.
\end{equation}

Figure~\ref{fig:2} shows the temporal trajectories of system~\ref{eq1} for the initial condition~\eqref{eq:initial2} and the parameter values listed in Table~\ref{tabmodel}, representative of a laboratory environment. The results suggest that, in this setting, both egg and larval populations can be suppressed by a relatively small parasitoid population. In contrast, such suppression is rarely achieved under field conditions, where environmental variability and other adverse factors diminish parasitism rates having a direct impact on the values of the Holling pairs $(a_1,h_1)$ and $(a_2,h_2)$. As a result, hosts and parasitoids generally persist in coexistence. Under these circumstances, high pest densities may still lead to outbreaks, highlighting the need for effective pest management interventions.

\begin{figure}[ht]
    \centering    \includegraphics[width=1\linewidth]{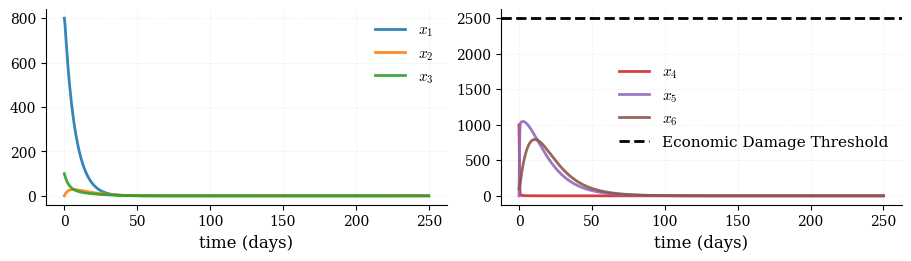}
    \vspace{-0.3cm}
    \caption{Simulation of the full six-dimensional model with initial condition~\eqref{eq:initial2}
and parameter values from Table~\ref{tabmodel}.
On the left: host egg compartments; un-parasitized eggs $x_1$ (blue),
parasitized eggs $x_2$ (orange), and egg-stage parasitoids $x_3$ (green),
in units of eggs~ha$^{-1}$.
On the right: host larval compartments; un-parasitized larvae $x_4$ (red),
parasitized larvae $x_5$ (purple), and larval-stage parasitoids $x_6$ (brown),
in units of larvae~ha$^{-1}$; the dashed horizontal line marks the economic
damage threshold of $2500$ larvae~ha$^{-1}$.
   }
    \label{fig:2}
\end{figure}

It is worth noting that substantial variability in the parameter values may arise even under laboratory conditions. For instance, the Holling parameters estimated in~\cite{Santos,Santos1} were $a_1=2.71, h_1=0.06$ and $a_2=0.28, h_2=1.46$, whereas experiments reported in~\cite{Lopes,Wiedenmann} yielded markedly different estimates: $a_1=0.02981, h_1=7.02651$ and $a_2=0.041, h_2=0.012$. Such discrepancies suggest that the dynamics of the sugarcane pest--parasitoid system are highly sensitive to biological and environmental factors, including temperature, humidity, interspecific interactions, and host-searching efficiency. Moreover, the observed variability may indicate potential parameter-identifiability issues, as well as strong correlations between the attack rate and handling time parameters within each Holling functional response.

To assess the effect of the potentially large variation in the Holling parameters, we perform a bifurcation analysis with respect to $h_2$ in~\eqref{eq1}, considering $h_2 \in [0.001,15]$ for two significantly different values of $a_2$, namely $a_2 = 4.1 \times 10^{-2}$ and $a_2 = 4.1 \times 10^{-5}$, while fixing all remaining parameters as in Table~\ref{tabmodel}. Figure~\ref{fig:4} illustrates the bifurcation behavior of the populations $x_1$, $x_4$, $x_5$, and $x_6$ with respect to $h_2$. Limit cycles arise for $h_2 \in [0.001,12]$ in both cases. The populations $x_2$ and $x_3$ converge to zero, and therefore their bifurcation diagrams are not shown. An analogous analysis with respect to $h_1$ revealed no bifurcation behavior.

\begin{figure}[ht]
    \centering
    \includegraphics[width=1\linewidth]{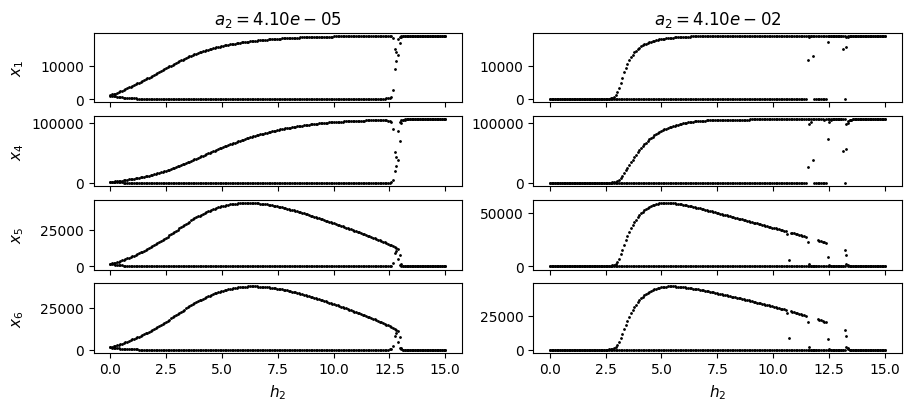}
    \caption{Bifurcation analysis for $h_2,$ with two different values of $a_2,$ namely $a_2 = 4.1 \times \times 10^{-5}$ on the left and $a_2 = 4.1 \times \times 10^{-2}$ on the right. Limit cases observed for $x_2$ and $x_3$ are disregarded, as they always approach the trivial state. Numerical experiments considered a $7000$ days horizon, with a transient time of $5000$ days. Initial condition given by \eqref{eq:initial2}. Further parameters presented in Table \ref{tabmodel}. }
    \label{fig:bif_analisys_complete}
\end{figure}

Figure~\ref{fig:4} shows the temporal trajectories of system~\eqref{eq1}
for $a_2 = 4.1 \times 10^{-5}$ and $h_2 = 6$, with initial
condition~\eqref{eq:initial2} and remaining parameters as in
Table~\ref{tabmodel}.
Under these values the system exhibits persistent oscillations, and the
resulting limit cycle is consistent with the bifurcation diagram in
Figure~\ref{fig:bif_analisys_complete}: the amplitude and period of the
oscillations match the branch identified in Figure~\ref{fig:bif_analisys_complete},
confirming that the bifurcation analysis correctly predicts the
long-term dynamical regime.

\begin{figure}[!ht]
    \centering    \includegraphics[width=1\linewidth]{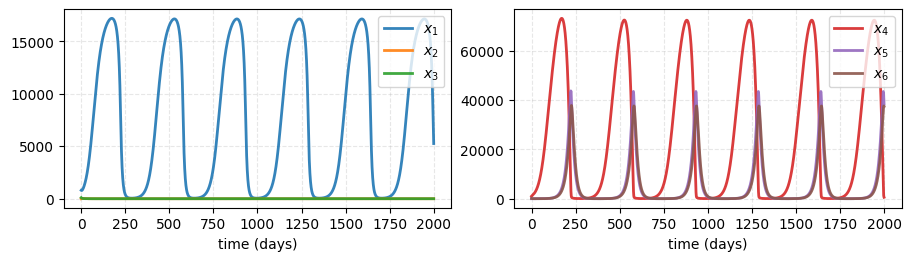}
    \caption{
    Temporal trajectories of system \eqref{eq1}. The parameter values are $a_2 = 4.1  \times 10^{-5}$ and $h_2=6.$ The initial condition is \eqref{eq:initial2}. Further parameters presented in Table \ref{tabmodel}. We observe limit cycles, and confirm that the values at the limit coincide with the ones shown in Figure \ref{fig:bif_analisys_complete}.
    }
    \label{fig:4}
\end{figure}

As illustrated in the sensitivity analysis in this section with respect to the parameters $a_2$ and $h_2$, the economic damage threshold is exceeded for a broad range of parameter values.
Indeed, Figures \ref{fig:1}, \ref{fig:bif_analisys_complete}, and \ref{fig:4} show that this threshold is exceeded under different environmental conditions, indicating the need to implement appropriate control strategies, which will be addressed in the following section.

\section{Control strategies for pest management}
\label{SecControl}

This section is devoted to the design and analysis of control strategies to reduce or eliminate the pest population. The numerical results of the previous section show that, without control, the pest populations (both eggs and larvae) rapidly reach levels well above the economic damage threshold, whereas sufficiently large releases of parasitoids can suppress the infestation but require unrealistic release sizes. Moreover, under weak parasitism efficiency (small $a_1$ or $a_2$, and large $h_1$ or $h_2$), the system exhibits persistent oscillations with pest levels remaining above the threshold. These findings motivate the search for effective and practically implementable control strategies.

\subsection{Optimal Control strategy}
\label{SecOC}

Assuming a continuous release rate of parasitoids in the environment, we consider the control system
\begin{equation}\label{controlsys}
\begin{split}
    \frac{d x_1}{dt} & = r \sigma_1 x_4 \left(1 - \frac{x_1}{K}\right)- m_1 x_1 - n_1 x_1 - \frac{a_1 x_1x_3}{1 + a_1h_1x_1},\\
    \frac{d x_2}{dt} &= \frac{a_1 x_1x_3}{1 + a_1h_1x_1} - m_2 x_2 - n_2 x_2,\\
    \frac{d x_3}{dt} &= \gamma_3  n_2 x_2 - m_3 x_3 + u,\\
    \frac{d x_4}{dt} &= n_1 x_1  - m_4 x_4 - n_3 x_4 - \frac{a_2x_4x_6}{1 + a_2h_2x_4} , \\
    \frac{d x_5}{dt} &= \frac{a_2 x_4x_6}{1 + a_2h_2x_4}  - m_5 x_5 - n_4 x_5,\\
    \frac{d x_6}{dt} &= \gamma_6 n_4 x_5 - m_6 x_6 + v, 
\end{split}
\end{equation}
where $u$ and $v$ denote the release rate per unit time of the egg parasitoid {\em T. galloi} and of the larval parasitoid {\em C. flavipes}, respectively.
For practical implementation, the controls are assumed to satisfy the bounds
\begin{equation}\label{eq:controlbounds}
0 \leq u(t) \leq u_{\max},\qquad 0 \leq v(t) \leq v_{\max},
\end{equation}
where $u_{\max}$ and $v_{\max}$ are given positive constants that represent the maximum feasible instantaneous release rates of parasitoids at any given time $t\ge 0$.
We consider the associated optimal control problem, which aims to minimize the cost function
\begin{equation}\label{eq:costfunct}
J(u,v) := \int_0^T \Big[
\beta_1 x_1(t) + \beta_4 x_4(t) + \beta_5 x_5(t)
+\alpha_u u(t) + \alpha_v v(t)
\Big] \, dt,
\end{equation}
over the horizon $[0,T]$ for some $T>0$, where $\beta_1,\beta_4,\beta_5,\alpha_u,$ and $\alpha_v$ are nonnegative weights associated with the pest populations and the parasitoid release efforts, respectively. The variable $x_5$ is included in the cost function because parasitized sugarcane borer larvae remain active and may still cause damage to the plantation \cite{Botelho,Santos}. 
We further assume that, above the already-mentioned economic damage threshold, the plantation becomes completely compromised. This leads to the state constraint
\begin{equation}
\label{stateconstraint}
x_4(t)+x_5(t) \leq L,
\qquad {\rm for \, all\,\, } t\in [0,T],
\end{equation}
where $L$ represents the maximum admissible larvae density---related to the economic damage threshold introduced before.
In addition, we impose the terminal constraint
\begin{equation}
\label{eq:finalconstraint}
x_1(T) + x_4(T) + x_5(T) \leq \varepsilon_T,
\end{equation}
which forces the pest population to end below a prescribed acceptable level. 
In summary, the optimal control problem under consideration is to minimize the cost function in~\eqref{eq:costfunct} subject to~\eqref{controlsys}-\eqref{eq:controlbounds}-\eqref{stateconstraint}-\eqref{eq:finalconstraint}.

The simulations for the optimal control problem \eqref{controlsys}-\eqref{eq:finalconstraint}, displayed in Figure~\ref{fig:ocp1}, illustrate the optimal control profiles and the corresponding state trajectories over the time horizon $T=70$ days. The parameter values are as reported in Table~\ref{tabmodel}, except for $a_2=4.1 \times 10^{-5}$ and $h_2=6$. We consider $L=2000$ larvae/ha in order to stay far from the economic damage threshold of $2500$ larvae/ha, and upper bounds on the controls given by $u_{\max}=v_{\max}=1000$. The weights in the cost function are set to $\alpha_u=\alpha_v=1$ and $\beta_1=\beta_4=\beta_5=10$, balancing the reduction of pest populations with the implementation costs. The terminal constraint \eqref{eq:finalconstraint} with $\varepsilon_T=25$ is satisfied by the optimal solution, which represents about $1\%$ of the economic damage threshold.
Starting from the initial condition in~\eqref{eq:initial1}, the optimal strategy leads to a total release of approximately $\int_0^{70} u(t)dt \approx 4584$ units of \emph{T.~galloi} and $\int_0^{70} v(t)dt \approx 17061$ units of \emph{C.~flavipes}. 
The optimal control simulations were carried out using Python - GEKKO package for Optimal Control \cite{beal2018gekko}.

In the following subsections, we present alternative control techniques, primarily formulated in closed-loop form. We conclude with a comparative analysis of the performance of the different strategies.

\begin{figure}
    \centering
    \includegraphics[width=1\linewidth]{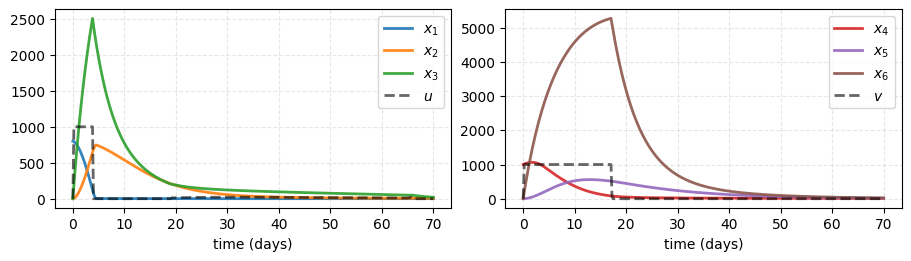}
    \caption{Optimal controls and state trajectories for  problem \eqref{controlsys}-\eqref{eq:finalconstraint}. The initial condition is given in~\eqref{eq:initial1}. 
Parameter values are as in Table~\ref{tabmodel}, except for $a_2=4.1 \times 10^{-5}$ and $h_2=6$. The time horizon is $T=70$ days, the upper bound in~\eqref{stateconstraint} is $L=2000$, and the upper bounds in~\eqref{eq:controlbounds} are $u_{\max}=1000$ and $v_{\max}=1000$. The weights in the integral cost function in~\eqref{eq:costfunct} are $\alpha_u=\alpha_v=1$ and $\beta_1=\beta_4=\beta_5=10$. The terminal constraint in~\eqref{eq:finalconstraint} is $\varepsilon_T=25$. The total number of released parasitoids is $\int_0^{70} u(t)\,dt \approx 4584$ units of \emph{T.~galloi} (dashed line in left figure) and $\int_0^{70} v(t)\,dt \approx 17061$ units of \emph{C.~flavipes} (dashed line in right figure).}
    \label{fig:ocp1}
\end{figure}

\subsection{State-Dependent Riccati Equation (SDRE) approach}\label{sec:SDRE}

In this subsection, we employ a control technique based on suboptimal feedback control, known as the {\em State-Dependent Riccati Equation (SDRE) method} \cite{MRACEKKK,MRACEK}. We begin by presenting the main elements of this technique.

\subsubsection{The State-Dependent Riccati Equation (SDRE) method}

In this section we briefly recall the SDRE method for a general control-affine system of the form
\begin{equation}
\label{eq2pont1}
\begin{array}{ll}
\dot{x}=f(x)+B(x)u, \quad f(0)=0, \\
y=C(x)x,
\end{array}
\end{equation}
where $x \in \mathbb{R}^{n}$ denotes the state vector, $u \in \mathbb{R}^{m}$ the control input, and $y \in \mathbb{R}^{s}$ the system output. The mappings $f:\mathbb{R}^{n}\to\mathbb{R}^{n},$  $B:\mathbb{R}^{n}\to\mathbb{R}^{n\times m}$ and
$C:\mathbb{R}^{n}\to\mathbb{R}^{s\times n}$
are assumed to be continuously differentiable with respect to their argument. Moreover, we consider the cost function
\begin{equation}
\label{eq2pont2}
J(u)= \int^{\infty}_{0}\left[ x^{\top}Q(x)x+ u^{\top}R(x)u \right]dt,
\end{equation}
where the matrix 
$Q(x)=C(x)^\top C(x)\in\mathbb{R}^{n\times n}$ is symmetric positive semidefinite and the matrix $R(x)\in\mathbb{R}^{m\times m}$ is symmetric positive definite. We assume that the function $ R : \mathbb{R}^{n}\to\mathbb{R}^{m \times m} $ is sufficiently smooth. 
The nonlinear system (\ref{eq2pont1}) can be rewritten in a form analogous to a linear system through a state-dependent factorization \cite{BANKS}:
\begin{equation}\label{eq2pont3}
\begin{array}{ll}
\dot{x}=A(x)x+B(x)u, \quad
y=C(x)x, \
\end{array}
\end{equation}
where $f(x)=A(x)x$.
Note that the choice of the matrix $A(x)$ is not unique. Different parameterizations may be adopted, provided that the detectability and stabilizability properties of the system, in the sense described below, are preserved.

To apply the SDRE method to the controlled system (\ref{eq2pont3}), the system must verify some algebraic hypotheses~\cite{MRACEKKK, MRACEK}. Under these assumptions, suboptimal solutions are obtained at each integration step of system~\eqref{eq2pont3}. More precisely, according to \cite[Theorem 1]{MRACEK}, the closed-loop solution obtained by the SDRE method (see \eqref{eq2pont6} below) is locally asymptotically stable provided that the pairs $\{A(x),B(x)\}$ and $\{C(x),A(x)\}$ are {\em pointwise stabilizable and detectable,} respectively, in the linear sense, for all $x\in \mathbb{R}^n$.

By applying Pontryagin's Maximum Principle to the optimal control problem \eqref{eq2pont2}-\eqref{eq2pont3} (see, {\em e.g.},~\cite{Sethi2000}), we obtain a closed-form for the optimal control:
\begin{equation}\label{eq2pont6}
u(x)=-R(x)^{-1}B(x)^{\top}P(x)x,
\end{equation}
where the matrix $P(x)$ is symmetric and positive definite for all values of $x$, and is the solution of the {\em Riccati equation} with state-dependent matrices
\begin{equation}\label{eq2pont5}
\displaystyle P(x)A(x)+A^{\top}(x)P(x)-P(x)B(x)R^{-1}(x)B^{\top}(x)P(x)+Q(x)=0.
\end{equation}

Such a matrix equation does not generally admit an analytic solution due to the state dependence of the involved matrices. For this reason, we rely on suitable numerical methods to determine a suboptimal solution to the control problem (\ref{eq2pont2}) - (\ref{eq2pont5}), according to the discrete-time Algorithm \ref{alg:dtsdre} described below (see also~\cite{Molter,Molter2010}).

\begin{algorithm}
\caption{Discrete-Time SDRE}
\label{alg:dtsdre}
\begin{algorithmic}[]

\State 
\noindent\textbf{Step 1:} Define the discrete-time nonlinear system in the state-dependent coefficient form
\[
x_{k+1} = A(x_k)x_k + B(x_k)u_k ,
\]
with sampling instant \(k \in \{0,\ldots,N-1\}\), $x_k = x(t_k)$, $0 = t_0 < \ldots < t_N = T$.

\State 
\noindent\textbf{Step 2:} For each sampling instant \(k \in \{0,\ldots,N-1\}\), measure the system state \(x_k\) and compute the matrices
\(
A(x_k), B(x_k),
\)
as well as the weighting matrices
\(
Q(x_k), R(x_k).
\)

\State 
\noindent\textbf{Step 3:} Solve the  state-dependent Riccati equation (\ref{eq2pont5})
and obtain \(P(x_k)\).

\State 
\noindent\textbf{Step 4:} Compute the control law $u(t)\equiv u_k $ for any $ t\in [t_k,t_{k+1}) $, where $u_k = u(t_k)$ is derived from equation (\ref{eq2pont6}).

\State 
\noindent\textbf{Step 5:} Propagate the system dynamics to obtain the next state
\[
x_{k+1}=A(x_k)x_k+B(x_k)u_k.
\]
Then, increment \(k \leftarrow k+1\) and return to \textbf{Step 2}.
\end{algorithmic}
\end{algorithm}

\subsubsection{SDRE method for pest control}
\paragraph{Case I - SDRE continuous release}\label{SDRE-continuous}
We now apply Algorithm \ref{alg:dtsdre} to the controlled system \eqref{controlsys}. We first recast
system~\eqref{controlsys} in the form~\eqref{eq2pont3} by introducing the following state-dependent matrices $A$, $B$ and $C$, as well as weight matrices $Q$ and $R$:
\begingroup
\small
\setlength{\arraycolsep}{1pt}
\begin{equation}\label{eq:matrixAx}
A(x)= \left[\begin{array}{cccccc}
-(m_1+n_1) & 0 & -\frac{a_1x_1}{1+a_1h_1x_1} & r\sigma_1\left(1-\frac{x_1}{K}\right) &0 & 0\\
\frac{a_1x_3}{1+a_1h_1x_1} & -(m_{2}+n_{2}) &  0 & 0 & 0 & 0\\
0 & \gamma_3n_2 & -m_3 & 0 & 0 & 0 \\
n_1 & 0 & 0 & -(m_4+n_3)& 0 & -\frac{a_2x_4}{1+a_2h_2x_4} \\
0 & 0 & 0 & \frac{a_2x_6}{1+a_2h_2x_4} & -(m_5+n_4) & 0 \\
0 & 0 & 0 & 0 & \gamma_6n_4 & -m_6
\end{array}\right],
\end{equation}
\endgroup

\begin{equation}\label{eq:matricesSDRE} 
B =
\begin{bmatrix}
0 & 0 \\
0 & 0 \\
1 & 0 \\
0 & 0 \\
0 & 0 \\
0 & 1
\end{bmatrix},\;\;
C=
{\left[\begin{array}{cccccc}
1&0&0&0&0&0\\
0&0&0&1&0&0\\
0&0&0&0&1&0
\end{array}
\right]},\;\;
R=
{\left[\begin{array}{cc}
25&0 \\
0&1
\end{array}
\right]},\;\;Q=C^\top C.
\end{equation}
This choice of matrix $C$ is motivated by the fact that only populations $x_1$, $x_4$, and $x_5$ are measurable \cite{Molter}, and these are precisely the variables we seek to reduce. More specifically, the state-dependent term in the cost function~\eqref{eq2pont2} penalizes the quantity $x_1^2 + x_4^2 + x_5^2$, so the optimization procedure aims to drive these three variables towards zero.

Once the matrices $A$, $B$, $C$, $Q$, and $R$ in~\eqref{eq:matrixAx}--\eqref{eq:matricesSDRE} are specified, we numerically verified, using symbolic computation, that the pairs $\{A(x),B(x)\}$ and $\{C(x),A(x)\}$ are pointwise stabilizable and detectable, respectively, as required. The verification code is included in the above-cited repository.
Afterwards, we determine the controls $u$ and $v$ in~\eqref{controlsys} following the steps described in Algorithm~\ref{alg:dtsdre}. 

Figure~\ref{figure1} illustrates the controlled trajectories and the corresponding control profiles. 
We use parameter values from Table~\ref{tabmodel} except $h_2 = 6$, $a_2 = 4.1 \times 10^{-5}$, and the initial conditions~\eqref{eq:initial1}. Recall that, under these values, the release of parasitoids into the system leads to the emergence of limit cycles, as shown in Figure~\ref{fig:4}. However, the application of the control is expected to drive the system toward stability, eliminating sustained oscillations over time, with the pest population tending to zero. Indeed, as illustrated in Figure~\ref{figure1}, the controlled state trajectories of system~\eqref{controlsys} with control functions \(u\) and \(v\) computed via Algorithm~\ref{alg:dtsdre} are controlled to zero.
Over the time horizon \(T=70\), the SDRE-based control strategy yields a terminal value \(x_1(T)+x_4(T)+x_5(T)\approx 53\), indicating a significant, although not complete, reduction of the targeted populations. The cumulative control effort corresponds to approximately \(\int_0^{70} u(t)\,dt \approx 2092\) units of \emph{T.~galloi} (green curve) and \(\int_0^{70} v(t)\,dt \approx 17049\) units of \emph{C.~flavipes} (blue curve).
The black dotted curve represents the function $x_4+x_5$ as a function of time. One can observe that this SDRE feedback control keeps the pest population far below the damage threshold of $2500$ larvae/ha at all times.

\begin{figure}[!ht]
\begin{center}
	\includegraphics[width=1\linewidth]{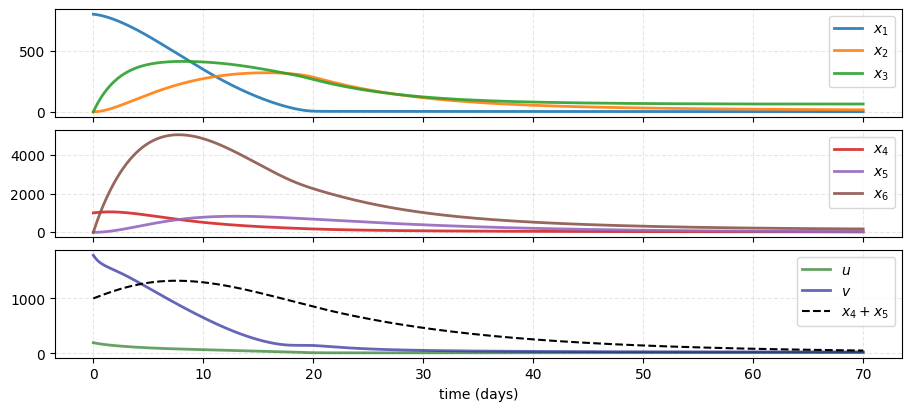 }
\caption{Temporal trajectories of populations $x_1$ to $x_6$ and control functions $u$, $v$ for system~\eqref{eq2pont3} with matrices from~\eqref{eq:matrixAx}-\eqref{eq:matricesSDRE} and control inputs computed by Algorithm~\ref{alg:dtsdre}, initial conditions \eqref{eq:initial1}, $a_2 = 4.1 \times 10^{-5}$,  $h_2 = 6$, and the other parameters from Table \ref{tabmodel}.
Time horizon is $T=70$.
The terminal value for selected state variables add to $x_1(T)+x_4(T)+x_5(T)\approx 53$. The total number of released parasitoids is $\int_0^{70} u(t)\,dt \approx 2092$ units of \emph{T.~galloi} (green curve in bottom figure) and $\int_0^{70} v(t)\,dt \approx 17049$ units of \emph{C.~flavipes} (blue curve in bottom figure).} 
\label{figure1}
\end{center}
\end{figure}

\paragraph{Case II - Event-triggered SDRE control}
In this section, we combine the SDRE method with an event-triggered approach. We consider 
an alert level of $2000$ larvae/ha, to ensure keeping the larvae population below the damage threshold of $2500$ larvae/ha~\cite{MOLNAR,PARRA}. Since only larvae cause damage to the crop among the six populations in Table~\ref{tab:populations}, we decide to activate the control by releasing parasitoids into the field {\em only} when the number of larvae exceeds the alert level; that is, whenever
\begin{equation}
\label{threshold}
    x_4 + x_5 > 2000.
\end{equation}
The simulations for this scenario are presented in Figure~\ref{Fig:SDRE_threshold}, which displays the temporal trajectories of the populations $x_1$ to $x_6$ together with the control inputs $u$ and $v$ for system~\eqref{eq2pont3}. The involved matrices are given in~\eqref{eq:matrixAx}--\eqref{eq:matricesSDRE}, and the controls are computed via Algorithm~\ref{alg:dtsdre} whenever the event-triggered condition~\eqref{threshold} is active. The simulations are performed with initial conditions~\eqref{eq:initial1} and parameters $h_2 = 6$, $a_2 = 4.1 \times 10^{-5}$, while the remaining parameters are taken from Table~\ref{tabmodel}. A longer time horizon $T=400$ is considered in order to highlight the limit cycle behavior of the dynamics. As observed in the figure, the control actions are activated only when the combined population $x_4 + x_5$ attains the alert level $L=2000$, confirming the effectiveness of the event-triggered strategy. Moreover, the quantity $x_4 + x_5$ reaches a minimum around day 82, within which the total number of released parasitoids is approximately $\int_0^{82} u(t)\,dt \approx 1405$ units of \emph{T.~galloi} (green curve in the bottom figure) and $\int_0^{82} v(t)\,dt \approx 21929$ units of \emph{C.~flavipes} (blue curve in the bottom figure).

\begin{figure}[!ht]
\begin{center}
	\includegraphics[width=1\linewidth]{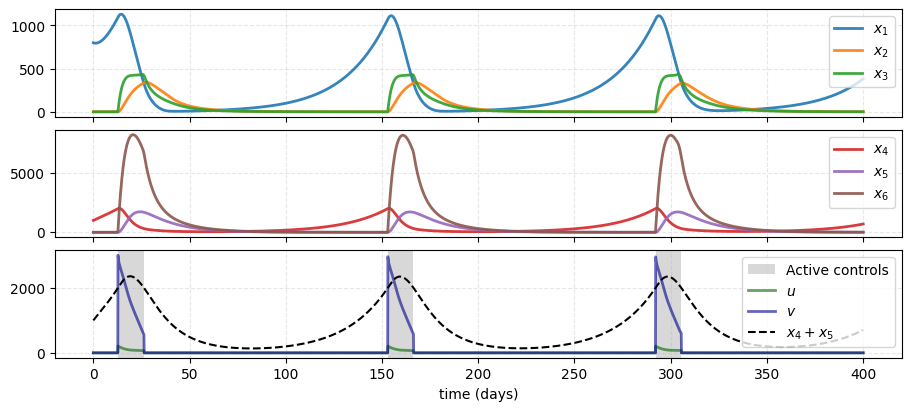}
\caption{
Temporal trajectories of populations $x_1$ to $x_6$ and controls $u$ and $v$  for system~\eqref{eq2pont3} with matrices from~\eqref{eq:matrixAx}-\eqref{eq:matricesSDRE} and control inputs computed by Algorithm~\ref{alg:dtsdre} whenever the event-triggered condition~\eqref{threshold} is active. Initial conditions \eqref{eq:initial1}, $a_2 = 4.1 \times 10^{-5}$,  $h_2 = 6$, and the other parameters from Table \ref{tabmodel}.
Time horizon is $T=400$ to emphasize the limit cycle behavior of the dynamics. 
The function $x_4+x_5$ attains a minimum around day 82. The total number of released parasitoids in the interval $[0,82]$ is $\int_0^{82} u(t)\,dt \approx 1405$ units of \emph{T.~galloi} (green curve in bottom figure) and $\int_0^{82} v(t)\,dt \approx 21929$ units of \emph{C.~flavipes} (blue curve in bottom figure). 
}
\label{Fig:SDRE_threshold}
\end{center}
\end{figure}
While the SDRE strategies perform well when continuous state monitoring is available, field conditions often permit only discrete, scheduled measurements and releases. We therefore turn next to a Lyapunov-based impulsive feedback approach designed precisely for this setting.

\color{black}
\subsection{Impulsive feedback control strategy} \label{impulsivo}

In this section, we propose a feedback stabilization strategy based on Lyapunov arguments combined with impulsive control actions. The main idea is to find a suitable Lyapunov function for a reduced subsystem capturing the key pest populations and to enforce a monotonic decrease of this function along system trajectories. We introduce state-dependent impulsive interventions acting on the populations of parasitoids. These impulses are designed through explicit threshold conditions ensuring that the derivative of the Lyapunov function remains strictly negative outside the origin. This approach leads to implementable feedback laws, activated at discrete times, which guarantee the reduction of the target populations with exponential decrease.

To this aim, we consider the
impulsive control system consisting of  model \eqref{eq1} for times $t$ in $ [0,T]\setminus \mathcal{T}$, where $\mathcal{T} := \{t_1,\ldots,t_N\}$ is the set of time instances of release of parasitoids in the environment, $N$ is the total number of releases, and the dynamics at $\mathcal{T}$ is given by
\begin{equation}\label{impcontrolsysT}
\begin{aligned}
    x_1(t_k^+) &= x_1(t_k^-) ,\quad
    x_2(t_k^+) = x_2(t_k^-),\quad
    x_3(t_k^+) = x_3(t_k^-) + \Delta_3(t_k),\quad t_k\in\mathcal{T},\\
    x_4(t_k^+) &= x_4(t_k^-),\quad 
    x_5(t_k^+) = x_5(t_k^-),\quad
    x_6(t_k^+) = x_6(t_k^-)  + \Delta_6(t_k),\quad
    t_k\in\mathcal{T},
\end{aligned}
\end{equation}
where $\Delta_3(t_k)$ and $\Delta_6(t_k)$ represent the amount of parasitoids released at time $t_k\in\mathcal{T}$. 

In order to obtain suitable values for the release quantities $\Delta_3(t_k)$ and $\Delta_6(t_k)$, we consider the {\em Lyapunov function}
\begin{equation}
\label{eq:Lyapfun}
    V(x) := \frac{1}{2}\left(x_1^2 + x_4^2\right)\, ,\qquad x = (x_1,\ldots,x_6)\in \mathcal{C}\,.
\end{equation}
The parasitized larval population $x_5$ is excluded from $V$ because it is
implicitly regulated through $x_4$: once un-parasitized larvae are suppressed,
the inflow to $x_5$ ceases and it decays to zero, so
$V = \tfrac{1}{2}(x_1^2 + x_4^2)$ is sufficient to certify convergence of all pest compartments. 
We get the following result:
\begin{proposition}
    \label{PropImpulses}
If one chooses $\Delta_j(t_k)$, for $j=3,6$,  satisfying
\begin{equation}\label{Deltas}
\Delta_j(t_k) \geq  G_j(x_1(t_k),x_4(t_k)) - x_j(t_k^-)\,,\quad j = 3,6\, ,
\end{equation}
for any time $t_k^-$, where
\begin{equation*}
    \begin{split}
    G_3(x_1,x_4) := \frac{1+a_1 h_1 x_1}{a_1 x_1}
   \left[ r\sigma_1\left(1 - \frac{x_1}{K}\right) + n_1\right] x_4\,,
   \\ 
G_6(x_1,x_4) := \frac{1+a_2 h_2 x_4}{a_2 x_4}
   \left[ r\sigma_1\left(1 - \frac{x_1}{K}\right) + n_1\right] x_1\,,
    \end{split}
\end{equation*}
then the Lyapunov function~\eqref{eq:Lyapfun} satisfies 
$\dot{V}(t_k)\le -\nu V(t_k)$ for all $k=1,\dots$, and for some $\nu>0$.
\end{proposition}

\begin{proof}
    Consider the Lyapunov function $V$ given in \eqref{eq:Lyapfun}. 
Its time derivative along the trajectories of the system satisfies
\begin{equation}
\label{Vleq0}
\begin{aligned}
    \frac{dV}{dt} &= x_1 \dot{x}_1 + x_4 \dot{x}_4 \\ 
    &\leq -\nu V 
    + \left[ r\sigma_1\left(1 - \frac{x_1}{K}\right) + n_1\right] x_1 x_4 
    - \frac{a_1 x_1^2 x_3}{1+a_1 h_1 x_1} 
    - \frac{a_2 x_4^2 x_6}{1+ a_2 h_2 x_4} \,,
\end{aligned}
\end{equation}
where $\nu := 2\min\{m_1+n_1,\; m_4+n_3\}$. Setting 
\[
\Xi:= \left[ r\sigma_1\left(1 - \frac{x_1}{K}\right) + n_1\right] x_1 x_4 
    - \frac{a_1 x_1^2 x_3}{1+a_1 h_1 x_1} 
    - \frac{a_2 x_4^2 x_6}{1+ a_2 h_2 x_4}\,,
\]
we aim to enforce $\Xi \leq 0$.
A sufficient condition for $\Xi \leq 0$ is
\begin{equation}
    \left[ r\sigma_1\left(1 - \frac{x_1}{K}\right) + n_1\right] x_1 x_4 
    \leq 
    \frac{a_1 x_1^2 x_3}{1+a_1 h_1 x_1}\,,
\end{equation}
which, for $x_1\neq 0$, is equivalent to the inequality
\begin{equation}\label{eq:impcond3}
   x_3 \geq G_3(x_1,x_4).
\end{equation}
Similarly, an alternative sufficient condition for $\Xi \leq 0$ is given by
\begin{equation}
    \left[ r\sigma_1\left(1 - \frac{x_1}{K}\right) + n_1\right] x_1 x_4  
    \leq 
    \frac{a_2 x_4^2 x_6}{1+ a_2 h_2 x_4}\,,
\end{equation}
which, for $x_4\neq 0$, yields the condition
\begin{equation}\label{eq:impcond6}
   x_6 \geq G_6(x_1,x_4).
\end{equation}
Note that $G_j(x_1,x_4)\ge 0,$ for any $x\in\mathcal{C}$, $j =3,6$. 
\end{proof}

The latter result motivates a {\em feedback impulsive control law}, which guarantees the exponential stability of the resulting closed-loop system.
More precisely, let $t_1=0$ and, for $k\ge 2$, define $t_k$ as the first time $t>t_{k-1}$ at which condition~\eqref{Deltas} is violated for some $j$. Then, it follows that $\dot V(t)\le -\nu V(t)$ for all $t\ge 0$, for some $\nu>0$. Consequently, the system~\eqref{eq1}--\eqref{impcontrolsysT} is exponentially stable.

The main limitation of Proposition \ref{PropImpulses} is that, in principle, it does not ensure a finite number of impulsive releases; that is, $\mathcal{T}$ may have infinitely many elements. For implementation purposes, we design the impulsive feedback law as follows: given a maximum number $N$ of releases over the time horizon $T$, we consider a release frequency $\tau = T/N$. Assuming that we can measure $x_1$ and $x_4$ at times $t_k := k\tau$ for any $k=0,\ldots ,N-1$, we define the impulsive feedback controls
\begin{equation}
\label{eq:Delta3}
    \Delta_3(k\tau) := \max \left\{ G_3\big(x_1(k\tau),x_4(k\tau)\big) - x_3(k\tau^-),\, 0 \right\}
\end{equation}
and
\begin{equation}
\label{eq:Delta6}
    \Delta_6(k\tau) := \max \left\{ G_6\big(x_1(k\tau),x_4(k\tau)\big) - x_6(k\tau^-),\, 0 \right\}.
\end{equation}
The simulations of system~\eqref{eq1}-\eqref{impcontrolsysT} with impulses \eqref{eq:Delta3}-\eqref{eq:Delta6} are presented in Figure~\ref{fig:impulsive_II}, where we display the state trajectories together with the control impulses. 
The numerical experiments are performed using the initial conditions in \eqref{eq:initial1}, with parameters $a_2 = 4.1\times 10^{-5}$ and $h_2 = 6$, and all remaining parameters from Table \ref{tabmodel}. The time horizon is $T=70$ and the release frequency is $\tau=7$ days.
Over the interval $[0,70]$, the total number of released parasitoids is given by $\int_0^{70} u(t),dt \approx 4678$ units of \emph{T.~galloi} (gray dotted curve in the left panel) and $\int_0^{70} v(t),dt \approx 6660$ units of \emph{C.~flavipes} (gray dotted curve in the right panel). The dynamics show that the aggregate population $x_1 + x_4 + x_5$ decreases to a final value of approximately $114$, while the peak value of $x_4 + x_5$ remains around $1109$, which is well below the economic damage threshold considered in the previous sections. These results indicate that the proposed impulsive feedback control strategy is effective in keeping the pest population under control while maintaining moderate levels of parasitoid release.
\begin{figure}[!ht]
    \centering
    \includegraphics[width=1\linewidth]{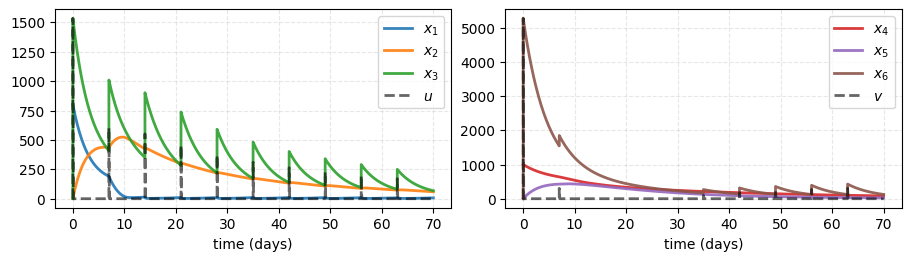}
    \caption{
    State trajectories and controls for the impulsive system  \eqref{eq1}-\eqref{impcontrolsysT}, with feedback impulses $\Delta_3(k\tau)$ and $\Delta_6(k\tau)$ given by~\eqref{eq:Delta3} and \eqref{eq:Delta6}, respectively.
    Initial conditions \eqref{eq:initial1}, $a_2 = 4.1\times 10^{-5}$,  $h_2 = 6$, and the other parameters from Table \ref{tabmodel}.
Time horizon is $T=70$ and release frequency $\tau=7.$  
 The total number of released parasitoids in the interval $[0,70]$ is $\int_0^{70} u(t)\,dt \approx 4678$ units of \emph{T.~galloi} (gray dotted curve in left figure) and $\int_0^{70} v(t)\,dt \approx 6660$ units of \emph{C.~flavipes} (gray dotted curve in right figure). The final value for $x_1+x_4+x_5$ is about 114 units. The maximum value of $x_4+x_5$ is approximately 1109, which is far from the economic damage threshold.
 }
\label{fig:impulsive_II}
\end{figure}

In the next section, we compare the different control strategies presented in this section and discuss other results of the article.

\section{Discussion}
\label{Sec:Discussion}

The occurrence of bifurcations in the model when varying the handling time of the larval parasitoid ($h_2$) reflects the transition between controlled laboratory stability and the complex oscillatory dynamics observed under field conditions. In environments with controlled temperature and the absence of physical barriers, the time the parasitoid spends to subdue and oviposit in the borer approaches its lower limit ($h_2 \approx 0$), which maximizes the per capita attack rate ($1/h_2$). However, the adverse conditions of the sugarcane ecosystem drastically alter this scenario. The decrease in temperature and the occurrence of rainy periods, typical of the sugarcane harvest season in Brazil, not only reduce the metabolic efficiency of the parasitoid but also intensify the defensive behavior of the host. This increase in biological and climatic difficulties expands the value of $h_2$ along the critical interval $[0.001, 12]$, acting as the mathematical trigger for the loss of fixed-point stability and the subsequent emergence of limit cycles through bifurcations.

This sensitivity of the model to variations in $h_2$ is grounded in the very nature of the foraging and management process of the parasitoid \textit{Cotesia flavipes} across the field-laboratory gradient. As pointed out by \cite{Wiedenmann}, parasitism efficiency estimates obtained in the laboratory frequently overestimate the natural enemy's reproductive success, as they disregard the time costs of searching in complex habitats. In the field, the success of biological control depends on a two-step search: habitat location (induced by plant volatiles) and borer location within the stalk, mediated by short-range cues such as the frass left by the insect. The severity of the environment and the vigor of the larvae in the field reduce effective encounter rates, severely inflating the handling time ($h_2$) as a function of climatic adversities and borer defenses.

In what respects to the control techniques presented in Section \ref{SecControl}, the three control strategies considered in this paper---optimal control in Section~\ref{SecOC}, SDRE-based feedback control in Section~\ref{sec:SDRE}, and impulsive feedback control in Section~\ref{impulsivo}---exhibit distinct trade-offs between control effort and effectiveness in suppressing the pest population.

The optimal control approach (Figure \ref{fig:ocp1}) achieves a strong reduction of the targeted populations while explicitly enforcing the terminal constraint \(x_1(T)+x_4(T)+x_5(T)\leq 25\), ensuring a very low final infestation level (approximately \(1\%\) of the economic damage threshold). This high level of performance, however, comes at the expense of a relatively large cumulative release of \emph{C.~flavipes} (about \(17061\) units), while the release of \emph{T.~galloi} remains moderate (approximately \(4584\) units).

The SDRE-based feedback strategy (Figure \ref{figure1} in Section~\ref{SDRE-continuous}) provides a comparable qualitative behavior in terms of maintaining the pest population below the economic threshold throughout the time horizon, with a final value \(x_1(T)+x_4(T)+x_5(T)\approx 53\), slightly higher than in the optimal control case. Interestingly, this method significantly reduces the required amount of \emph{T.~galloi} (about \(2092\) units), while keeping a similar level of \emph{C.~flavipes} release (approximately \(17049\) units). 
Notice that the SDRE formulation does not explicitly enforce the control bounds~\eqref{eq:controlbounds}, the state constraint~\eqref{stateconstraint}, as well as the final time constraint~\eqref{eq:finalconstraint} of the optimal control framework. 
Hence, SDRE-based feedback  offers a robust and  efficient use of the control agents.

In contrast, the impulsive feedback control (Figure \ref{fig:impulsive_II}) leads to a higher final value of the aggregate population (around \(114\)), indicating a less aggressive suppression compared to the previous two strategies. Nevertheless, it still guarantees that the pest population remains well below the economic damage threshold at all times, with a peak value of \(x_4+x_5\) around \(1109\). A key advantage of this approach lies in its substantially lower cumulative release of \emph{C.~flavipes} (approximately \(6660\) units), combined with a moderate use of \emph{T.~galloi} (about \(4678\) units). Moreover, the impulsive nature of the control, with interventions at discrete times, may be more practical for real-world implementation.

The SDRE event-triggered strategy of Figure \ref{Fig:SDRE_threshold} is not directly comparable to the other approaches, since it is designed with a different objective. In particular, the control action is activated only when the state exceeds a prescribed threshold, whereas the other strategies determine release actions that aim at reducing a set of desired states at all times.

The performance comparison is summarized in Table \ref{tabelacontrole}.
Overall, the optimal control strategy provides the strongest suppression but at a higher cost in terms of control effort. The SDRE feedback offers a good compromise between performance and reduction of \emph{T.~galloi} releases, while the impulsive strategy stands out for its practical implementation and significantly reduced total releases, at the expense of a less stringent final suppression.

\begin{table}[ht]
\caption{Comparison of the three control strategies over the time horizon $T=70$, with initial conditions~\eqref{eq:initial1}, and parameter values $a_2=4.1 \times 10^{-5}$ and $h_2=6$, and all other remaining parameters from Table~\ref{tabmodel}.}\label{tabelacontrole}
\begin{tabular}{lcccc}
\toprule
\textbf{Technique} & \textbf{Final value} $(x_1+x_4+x_5)(T)$ & $\int_0^{T} u(t)\,dt$ & $\int_0^{T} v(t)\,dt$ & \textbf{Max} $(x_4+x_5)$ \\
\midrule
Optimal Control  & $\approx 25$ (constraint) & $\approx 4584$ & $\approx 17061$ & $\approx 1147$ \\
SDRE Feedback  & $\approx 53$ & $\approx 2092$ & $\approx 17049$ & $\approx 1322$ \\
Impulsive Feedback  & $\approx 114$ & $\approx 4678$ & $\approx 6660$ & $\approx 1109$ \\
\botrule
\end{tabular}
\end{table}

The investigation of control strategies for {\it D. saccharalis} continue to evolve. In~\cite{Gomes}, the authors showed that increasing the number of releases from three to five at weekly intervals expands temporal coverage and enhances synchronization with the host’s short egg stage, leading to greater biological control effectiveness. To show the consistency of this study with field practices, two impulsive control simulations were carried out and compared with current practices in sugarcane fields, as presented in \cite{Gomes}. The first simulation considered releases over five consecutive weeks, using the impulsive feedback formulae \eqref{eq:Delta3}-\eqref{eq:Delta6}, without imposing a fixed number of parasitoids. The corresponding trajectories are illustrated in Figure~\ref{fig:impulsive_feedback_5releases}. 
\begin{figure}
    \centering
    \includegraphics[width=\linewidth]{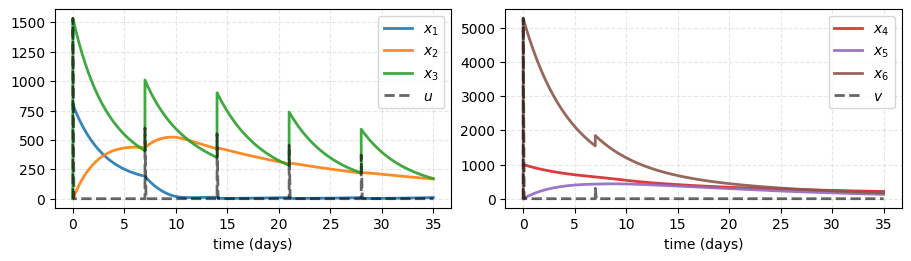}
    \caption{State trajectories and controls for the impulsive system \eqref{eq1}-\eqref{impcontrolsysT}, with feedback impulses $\Delta_3(n\tau)$ and $\Delta_6(n\tau)$ given by \eqref{eq:Delta3} and \eqref{eq:Delta6}, respectively.
    Initial conditions \eqref{eq:initial1}, $a_2 = 4.1\times 10^{-5}$,  $h_2 = 6$, and the other parameters from Table \ref{tabmodel}.
Time horizon is $T=35$, since we fix the total number of releases to $5$ and the release frequency to $\tau=7$ days. 
The total number of released parasitoids in the interval $[0,35]$ is $ \approx 3512 $ units of \emph{T.~galloi} (gray dotted curve in left figure) and $\approx 5578$ units of \emph{C.~flavipes} (gray dotted curve in right figure). The final value for $x_1+x_4+x_5$ is about 348 units.}
\label{fig:impulsive_feedback_5releases}
\end{figure}
The second simulation also considered releases over five consecutive weeks, but with a fixed amount of $\Delta_3(0)$ for the {\em T. galloi} and $\Delta_6(0)$ for the {\em C. flavipes} parasitoids released each week, with a total of 5 releases, according to~\cite{Gomes}. The resulting trajectories of the system are presented in Figure~\ref{fig:impulsive_constant_5releases}.

\begin{figure}
    \centering
    \includegraphics[width=\linewidth]{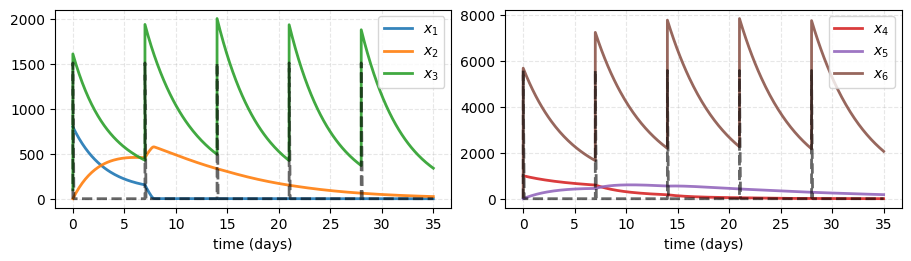}
    \caption{
    State trajectories and controls for the impulsive system \eqref{eq1}-\eqref{impcontrolsysT}, with {\bf constant} impulse magnitude $\Delta_3(0)$ and $\Delta_6(0),$ as defined in \eqref{eq:Delta3} and \eqref{eq:Delta6}, respectively.
    Initial conditions~\eqref{eq:initial1}, $a_2 = 4.1\times 10^{-5}$,  $h_2 = 6$, and the other parameters from Table \ref{tabmodel}.
The time horizon is $T=35$, since we fix the total number of releases to $5$ and the release frequency to $\tau=7.$  
 The total number of released parasitoids in the interval $[0,35]$ is $ \approx 7565 $ units of \emph{T.~galloi} (gray dotted curve in left figure) and $\approx 27895$ units of \emph{C.~flavipes} (gray dotted curve in right figure). The final value for $x_1+x_4+x_5$ is about 185 units.}
    \label{fig:impulsive_constant_5releases}
\end{figure}

Comparing the two simulations illustrated in Figures \ref{fig:impulsive_feedback_5releases} and \ref{fig:impulsive_constant_5releases}, we observed that, although the system is controlled in both cases,
the releases calculated over five weeks using the feedback impulsive control described in Section \ref{impulsivo} and determined without imposing a fixed number of parasitoids to be released, may result in significant savings in parasitoid production.

\section{Conclusion and perspectives}\label{SecConclusion}

In this study, a host-parasitoid system was considered to describe the behavior and biological control of the sugarcane borer by means of two of its parasitoids: {\it T. galloi} for the egg stage, and {\it C. flavipes} for the larval stage. The interaction between host populations and their parasitoids was modeled using a Holling Type II functional response, as experimental results indicated saturation in parasitism.

In terms of dynamic behavior, the system exhibited stable equilibrium points as well as limit cycles, depending on the values of the parameters of the Holling pairs $(a_1,h_1)$ and $(a_2,h_2)$.

We determined the five equilibrium points of the system, and performed a stability analysis of the extinction equilibrium, considered a potential objective in the context of biological pest control. However, both the coexistence equilibrium and the dynamics of the system in the presence of limit cycles entail a level of the pest population well above the economic damage threshold, which makes necessary to consider control strategies to mitigate the pest population below acceptable levels.

Three strategies were considered for biological pest control: optimal control, SDRE, and impulsive control. The numerical simulations indicate that, although continuous controls yield theoretically smoother trajectories, the impulsive control strategy emerges as the most feasible approach from both logistical and economic perspectives for sugarcane producers. In particular, it requires even smaller quantities of released parasitoids while still achieving highly satisfactory results. Quantitative comparisons among the different control strategies are presented in Table~\ref{tabelacontrole}.

We expect this study to serve as a basis for new approaches in the use of biological control and provide a decision-support tool that can assist biofactories and producers in optimizing the number of parasitoids to be released, thus reducing operational costs as well as the use of chemical pesticides.

Several directions remain open for future investigation.
The present model assumes constant environmental conditions, but sugarcane production (in Brazil and other locations) is strongly seasonal: temperature, rainfall and plant
phenology all vary over the growing cycle and modulate parasitoid attack rates and host development times.
Extending the framework to a periodically forced or Markov-switching system would allow the bifurcation structure identified here --- in particular the critical handling-time range for $h_2$ --- to be re-examined under realistic agronomic calendars.
A second priority is the experimental validation of the impulsive feedback
law~\eqref{eq:Delta3}--\eqref{eq:Delta6}: the release thresholds and
dosages prescribed by the control law should be tested against field-release records.
Finally, a formal parameter identifiability analysis for the Holling Type~II
pairs $(a_i, h_i)$, $i = 1, 2$, is needed before the model can be reliably
fitted to time-series data: structural identifiability determines whether
both the attack rate and the handling time can in principle be recovered from
population counts, while practical identifiability quantifies the
data quality and experimental design required to estimate them with sufficient
precision for control design.

Together, the six-dimensional host--parasitoid model, the bifurcation
analysis of the handling-time parameters, and the three complementary
control strategies --- optimal, SDRE feedback, and impulsive Lyapunov-based
--- establish a rigorous mathematical foundation for data-driven biological
control of \textit{D. saccharalis}, and raise concrete questions about
seasonal robustness, field validation, and parameter identifiability that
we intend to pursue in future work.

\backmatter

\section*{Declarations}

\subsection*{Funding}

This work was carried out within the framework of two collaborative projects: “Controle biológico de insetos-praga: modelagem matemática da dinâmica de sistemas hospedeiros-parasitoides e otimização da produção massal de parasitoides em laboratório”, funded by the FGV Research Network and coordinated by M.S.A., and NSERC Alliance Grant ALLRP 599133-24, coordinated by R.G. and funded by the Natural Sciences and Engineering Research Council of Canada (NSERC). Most of the research was conducted while F.A. was a Master's student at FGV EMAp supported by CAPES, and B.L. was a postdoctoral fellow at the same institution. M.S.A. was additionally supported by FAPERJ grants E-26/210.037/2024 and 260003/020457/2025, and by CNPq grant 312407/2023-8. R.G. also acknowledges support from NSERC through Discovery Grant RGPIN-2021-02632. A.M. was supported by the Research Support Foundation of the State of Rio Grande do Sul (FAPERGS), grant number 24/2551-0001559-0.

\subsection*{Conflict of Interest}

The authors declare that they have no conflicts of interest.

\subsection*{Code Availability}

All simulation, analysis, and optimization codes supporting this study's discussion are publicly available in the following GitHub repository: \url{https://github.com/fredsonaguiar/parasitoid-sugarcane-control}.

\subsection*{Author Contribution}

All authors contributed to the conceptualization, mathematical modeling, and control design.
Numerical experiments and optimization were performed by F.A., M.S.A., and A.M.;
mathematical analysis were performed by M.S.A., R.G., B.L., and A.M.;
the first draft of the manuscript was written by M.S.A., R.G., B.L. and A.M., and all authors commented on previous versions of the manuscript.
All authors read and approved the final manuscript.

\bibliography{bibliography}

\appendix

\section{Equilibria of the system - Proof of Proposition~\ref{prop:equilibria}}
\label{AppEqui}

\begin{proof}
Equilibrium points are determined as steady-states of system~\eqref{eq1}, that is, solutions $E^* = (x_1^*, x_2^*,x_3^*, x_4^*,x_5^*, x_6^*)$ of the system
\begin{equation}\label{eq:steadystatesystem}
\left\{
\begin{array}{l}
\displaystyle r \sigma_1 x_4 \left(1 - \frac{x_1}{K} \right) - (m_1 + n_1)x_1 - \frac{a_1 x_1 x_3}{1 + a_1 h_1 x_1} =0 \\[1ex]
\displaystyle \frac{a_1 x_1 x_3}{1 + a_1 h_1 x_1} - (m_2 + n_2)x_2 =0\\[1ex]
\displaystyle \gamma_3 n_2 x_2 - m_3 x_3 =0\\[1ex]
\displaystyle  n_1 x_1 - (m_4 + n_3)x_4 - \frac{a_2 x_4 x_6}{1 + a_2 h_2 x_4} =0\\[1ex]
\displaystyle
\frac{a_2 x_4 x_6}{1 + a_2 h_2 x_4} - m_5 x_5 - n_4 x_5 =0\\[1ex]
\displaystyle
\gamma_6 n_4 x_5 - m_6 x_6=0
\end{array}
\right.
\end{equation}
In the following, we denote {with a star $*$ the} steady state components. From the third and sixth equations, we get that
\begin{equation}\label{eq:x2x3x5x6}
x_3^* = \frac{\gamma_3 n_2}{m_3} x_2^*\, ,\qquad
x_6^* = \frac{\gamma_6 n_4}{m_6} x_5^*\,,
\end{equation}
respectively.
Plugging them into the second and fifth equations in \eqref{eq:steadystatesystem}, we get
\begin{align}\label{eq:beforecase1a}
0 = \left[\frac{a_1 x_1^*}{1 + a_1 h_1 x_1^*} \frac{\gamma_3 n_2}{m_3} - (m_2 + n_2)\right] x_2^*\, ,\qquad
\end{align}
and
\begin{align}\label{eq:beforecase1}
0 = \left[ \frac{a_2 x_4^*}{1 + a_2 h_2 x_4^*} \frac{\gamma_6 n_4}{m_6}  - (m_5 + n_4) \right] x_5^*\,,
\end{align}
respectively, which imply one of the following four cases.\newline
{\bf Case 1:} If \(x_2^* = x_5^* = 0\), then~\eqref{eq:x2x3x5x6} implies that \(x_3^* = x_6^* = 0\). From the fourth equation in~\eqref{eq:steadystatesystem} we deduce that
\begin{equation}\label{eq:x1x4}
x_4^* = \frac{n_1}{m_4 + n_3} x_1^*\, .
\end{equation}
{\bf Case 1.1:} If $x_1^* = 0$, we get the trivial equilibrium point $E^0 = (0,0,0,0,0,0)$.\newline
{\bf Case 1.2:} If $x_1^* \neq 0$, we determine the parasitoid-free equilibrium $E^{Pf} = (x_1^{Pf},0,0,x_4^{Pf},0,0)$, where from the first equation in~\eqref{eq:steadystatesystem} and~\eqref{eq:x1x4} we obtain that
\begin{equation}
\label{eqPf}
x_1^{Pf} = K \left[1 - \frac{(m_1 + n_1)(m_4 + n_3)}{r\sigma_1 n_1}\right]\, ,\qquad
x_4^{Pf} = \frac{n_1}{m_4 + n_3} x_1^{Pf}\, .    
\end{equation}
Notice that $x_1^{Pf} < K$ and so it is biologically feasible provided that Condition~\ref{assump:equilibria}.\ref{itma} is satisfied.\newline
{\bf Case 2:} If \(x_2^* = 0\) and \(x_5^*\neq 0\), then from~\eqref{eq:x2x3x5x6} and~\eqref{eq:beforecase1} we get that $x_3^*=0$ and
\begin{equation}\label{eq:case2x4}
x_4^* = \frac{(m_5 + n_4) m_6}{a_2 \gamma_6 n_4 - (m_5 + n_4) a_2 h_2 m_6}\,,
\end{equation}
respectively, provided that $a_2 \big(\gamma_6 n_4 - (m_5 + n_4) h_2 m_6\big) \neq 0$. Moreover, Condition~\ref{assump:equilibria}.\ref{itmb} ensures that the state $x_4^*$ is biologically interpretable. 
Then, from the first equation in~\eqref{eq:steadystatesystem} we get that
\begin{equation}\label{eq:case2x1}
x_1^* = K\,\frac{r\sigma_1 x_4^*}{r\sigma_1 x_4^* + K(m_1 + n_1)}\,,
\end{equation}
which implies $x_1^*<K$. From the fourth equation of~\eqref{eq:steadystatesystem} we get an expression for $x_6^*$, namely
\[
x_6^* = \left[n_1 x_1^* - (m_4 + n_3)x_4^*\right]\frac{1 + a_2 h_2 x_4^*}{a_2 x_4^*}\,,
\]
which makes biological sense whenever \(n_1 x_1^* - (m_4 + n_3)x_4^* \ge 0\), with $x_1^*$ and $x_4^*$ given by~\eqref{eq:case2x1} and~\eqref{eq:case2x4}, respectively. 
We impose strict positivity in the latter inequality, since otherwise we would get $x_6^*=0$ and thus $x_5^* =0$ from the sixth equations in~\eqref{eq:steadystatesystem}, which is analyzed in {\bf Cases~1} and {\bf 3}. Notice that the inequality \(n_1 x_1^* - (m_4 + n_3)x_4^* > 0\), with $x_1^*$ and $x_4^*$ given by~\eqref{eq:case2x1} and~\eqref{eq:case2x4}, respectively, can be expressed in terms of the system's parameters, and takes the form of Condition~\ref{assump:equilibria}.\ref{itmc}.
Finally, from the sixth equation in~\eqref{eq:steadystatesystem}, we can write $x_5^*$ in terms of $x_6^*$. Thus, under Conditions~\ref{assump:equilibria}.\ref{itmb}-\ref{itmc}, we determine the equilibrium in the absence of population of adult egg parasitoids \emph{T. galloi}, given by $E^{Cf}  = (x_1^{Cf},0,0,x_4^{Cf},x_5^{Cf},x_6^{Cf})$ with positive
\[
x_1^{Cf} = \frac{K r \sigma_1 x_4^{Cf}}{r\sigma_1 x_4^{Cf} + K(m_1 + n_1)}\,,\qquad 
x_4^{Cf} = \frac{(m_5 + n_4) m_6}{a_2 \gamma_6 n_4 - (m_5 + n_4) a_2 h_2 m_6}\,,
\]
\[
x_6^{Cf} = \left[n_1 x_1^{Cf} - (m_4 + n_3)x_4^{Cf}\right]\frac{1 + a_2 h_2 x_4^{Cf}}{a_2 x_4^{Cf}}\,,
\qquad x_5^{Cf} = \frac{m_6}{\gamma_6 n_4} x_6^{Cf}\, .
\]
{\bf Case 3:} If \(x_2^* \neq 0\) and \(x_5^*= 0\), then from~\eqref{eq:x2x3x5x6} and~\eqref{eq:beforecase1a} we get that $x_6^*=0$ and
\begin{equation}\label{eq:case1x1}
x_1^* = \frac{(m_2 + n_2) m_3}{a_1 \big( \gamma_3 n_2 - (m_2 + n_2) m_3 h_1 \big)}\,, 
\end{equation}
which is biologically reasonable provided Condition~\ref{assump:equilibria}.\ref{itmd} 
and
\begin{equation}
\label{pos6}
\frac{(m_2 + n_2) m_3}{a_1 \big( \gamma_3 n_2 - (m_2 + n_2) m_3 h_1 \big)} \le K\, .
\end{equation}
From the fourth equation in~\eqref{eq:steadystatesystem} we get that
\begin{equation}\label{eq:case3x1x4}
x_4^* = \frac{n_1}{m_4 + n_3}x_1^*,
\end{equation}
and from the first equation in~\eqref{eq:steadystatesystem} we deduce an expression for $x_3^*$ in terms of $x_1^*$ and $x_4^*$:
\begin{equation}\label{eq:relationx3withx1x4}
x_3^* = \left[ r\sigma_1 x_4^* \left(1 - \frac{x_1^*}{K} \right) - (m_1 + n_1)x_1^* \right] \frac{1 + a_1 h_1 x_1^*}{a_1 x_1^*} \, ,   
\end{equation}
which makes biological sense whenever 
\begin{equation}
    \label{pos3}
    x_1^* < K\,\frac{r \sigma_1 x_4^* }{r \sigma_1 x_4^* + K(m_1 + n_1)}\,, 
\end{equation}
for $x_1^*$ and $x_4^*$ given in \eqref{eq:case1x1} and \eqref{eq:case3x1x4}, respectively. Using the expressions of $x_1^*$ and $x_4^*$ for the current case, the inequality~\eqref{pos3} imposes a constraint on the model parameters, given by Condition~\ref{assump:equilibria}.\ref{itme}. 
Moreover, notice that $x_1^*$ in~\eqref{pos3} always satisfies the condition $x_1^* < K$, so in particular Condition~\ref{assump:equilibria}.\ref{itme} implies~\eqref{pos6}. 
So, under Conditions~\ref{assump:equilibria}.\ref{itmd}-\ref{itme}, we determine the equilibrium point of the system in the absence of population of adult larvae parasitoids \emph{C. flavipes}, given by $E^{Tg}  = (x_1^{Tg},x_2^{Tg},x_3^{Tg},x_4^{Tg},0,0)$ with {positive}
\[
x_1^{Tg} = \frac{(m_2 + n_2) m_3}{a_1 \gamma_3 n_2 - (m_2 + n_2) a_1 h_1 m_3}\, ,\qquad x_4^{Tg} = \frac{n_1}{m_4 + n_3} x_1^{Tg}\,,
\]
\[
x_3^{Tg} = \left[ r\sigma_1 x_4^{Tg} \left(1 - \frac{x_1^{Tg}}{K} \right) - (m_1 + n_1)x_1^{Tg} \right] \frac{1 + a_1 h_1 x_1^{Tg}}{a_1 x_1^{Tg}}\, , \qquad
x_2^{Tg} = \frac{m_3} {\gamma_3 n_2}x_3^{Tg}\, .
\]
{\bf Case 4:} If \(x_2^* \neq 0\) and \(x_5^*\neq 0\), then from~\eqref{eq:beforecase1a} and~\eqref{eq:beforecase1} we get that
\begin{equation}\label{eq:case4x1}
x_1^* = \frac{(m_2 + n_2) m_3}{a_1 \big[ \gamma_3 n_2 - (m_2 + n_2) m_3 h_1 \big]} 
\end{equation}
and
\begin{equation}\label{eq:case4x4}
x_4^* = \frac{(m_5 + n_4) m_6}{a_2 \big[\gamma_6 n_4 - (m_5 + n_4) h_2 m_6\big]}\,,
\end{equation}
respectively, under Conditions~\ref{assump:equilibria}.\ref{itmb}-\ref{itmd} and~\eqref{pos6}. From the first and fourth equations in~\eqref{eq:steadystatesystem} we derive an expression of $x_3$ and $x_6$ in terms of $x_1$ and $x_4$, and obtain that
\begin{align*}
x_3^* = \left[ r\sigma_1 x_4^* \left(1 - \frac{x_1^*}{K} \right) - (m_1 + n_1)x_1^* \right] \frac{1 + a_1 h_1 x_1^*}{a_1 x_1^*}
\end{align*}
and
\begin{align*}
x_6^* = \left[n_1 x_1^* - (m_4 + n_3)x_4^*\right]\frac{1 + a_2 h_2 x_4^*}{a_2 x_4^*}\, ,
\end{align*}
respectively. As before in {\bf Cases 2} and {\bf 3}, these expressions are biologically feasible provided that the terms in the square brackets are positive, which is ensured by Conditions~\ref{assump:equilibria}.\ref{itmf}-\ref{itmg}. 
In particular, the positivity of $x_3^*$ implies that $x_1^*$ must satisfy~\eqref{pos3}, which ensures that $x_1^* < K$, and so we can drop the assumption~\eqref{pos6} since Condition~\ref{assump:equilibria}.\ref{itmf} imposes a stronger condition on the system's parameters. 
So, under Conditions~\ref{assump:equilibria}.\ref{itmb}-\ref{itmd}-\ref{itmf}-\ref{itmg},
we determine the co-existence equilibrium $E^{\rm co} = (x_1^{co},x_2^{co},x_3^{co},x_4^{co},x_5^{co},x_6^{co})$,  with
\[
x_1^{co} = \frac{(m_2 + n_2) m_3}{a_1 \left[\gamma_3 n_2 - (m_2 + n_2) h_1 m_3\right]}\, ,\qquad 
x_4^{co} = \frac{(m_5 + n_4) m_6}{a_2 \left[ \gamma_6 n_4 - (m_5 + n_4) h_2 m_6\right]}\,,
\]
\[
x_3^{co} = \left[ r\sigma_1 x_4^{co} \left(1 - \frac{x_1^{co}}{K} \right) - (m_1 + n_1)x_1^{co} \right] \frac{1 + a_1 h_1 x_1^{co}}{a_1 x_1^{co}} \,,
\qquad x_2^{co} = \frac{m_3}{\gamma_3 n_2} x_3^{co}\, ,
\]
\[
x_6^{co} = \left[n_1 x_1^{co} - (m_4 + n_3)x_4^{co}\right]\frac{1 + a_2 h_2 x_4^{co}}{a_2 x_4^{co}} \, , \qquad
x_5^{co} = \frac{m_6}{\gamma_6 n_4} x_6^{co}\, . 
\]
In particular, $x_i^{co} > 0$ for all $i = 1,\ldots,6$, and $x_1^{co} < K$. 
\end{proof}

\section{Stability of Equilibria - Proof of Proposition~\ref{prop:staborigin}}
\label{appendixB}

\begin{proof}
The Jacobian matrix of system~\eqref{eq1} has the expression
\begin{equation} \label{eq:jacobian}
    \hat{J} = \begin{pmatrix}
        \hat{J_1} \, | \, \hat{J_2}
    \end{pmatrix}
\end{equation}
where
\begin{equation} \label{eq:jacobian1}
    \hat{J_1} = \begin{pmatrix}
        -\frac{r \sigma_1 x_4}{K} - m_1 - n_1-\frac{a_1 x_3}{(1 + a_1h_1x_1)^2}  & 0 & -\frac{a_1x_1}{1+a_1h_1x_1}\\
         \frac{a_1 x_3}{(1 + a_1h_1x_1)^2}  &  -m_2-n_2 & \frac{a_1x_1}{1+a_1h_1x_1}\\
         0 &  \gamma_3 n_2 & -m_3\\
         n_1 & 0 & 0\\
         0 & 0 & 0\\
         0 & 0 & 0
    \end{pmatrix}
\end{equation}
and
\begin{equation} \label{eq:jacobian2}
{\hat{J_2}=\begin{pmatrix}
r \sigma_1 \left(1 - \frac{x_1}{K}\right) & 0 & 0 \\
0 & 0 & 0 \\
0 & 0 & 0 \\
-m_4 -n_3 - \frac{a_2x_6}{(1 + a_2h_2x_4)^2} & 0 & -\frac{a_2x_4}{1 + a_2h_2x_4} \\
\frac{a_2x_6}{(1 + a_2h_2x_4)^2} & -m_5-n_4 & \frac{a_2x_4}{1 + a_2h_2x_4}\\
0 & \gamma_6  n_4 & -m_6
\end{pmatrix}}\;.
\end{equation}
Then the Jacobian matrix at $E^0 = (0,0,0,0,0,0)$ is
\begin{equation} \label{eq:jacobianatorigin}
\hat{J}(E^0)=\begin{pmatrix}
- m_1 - n_1  & 0 & 0 &  r \sigma_1  & 0 & 0 \\
 0  &  -m_2-n_2 & 0 & 0 & 0 & 0 \\
 0 &  \gamma_3
 n_2 & -m_3 & 0 & 0 & 0 \\
 n_1 & 0 & 0 & -m_4 -n_3  & 0 & 0 \\
 0 & 0 & 0 & 0 & -m_5-n_4 & 0 \\
 0 & 0 & 0 & 0 & \gamma_6 n_4 & -m_6
    \end{pmatrix}.
\end{equation}
A direct computation shows that $\hat{J}(E^0)$ has the following eigenvalues:
\[
\lambda_1 = -m_6,\qquad \lambda_2 = -m_5 - n_4,\qquad \lambda_3 = -m_3,\qquad \lambda_4 = -m_2 - n_2,
\]
\[
\lambda_5 = \frac{-b -\sqrt{b^2 - 4c}}{2}\;,\qquad \lambda_6 = \frac{-b +\sqrt{b^2 - 4c}}{2}\;,
\]
where
\[
b= m_1 + n_1 + m_4 + n_3,\qquad c = (m_1 + n_1)(m_4 + n_3) - r n_1 \sigma_1.
\]
It is readily observed that the first five eigenvalues are always negative under the natural assumption that the parameters of system~\eqref{eq1} are positive. On the other hand, $\lambda_6$ is negative if and only if $c > 0$. This prescribes a condition on the parameters of the system, namely that $(m_1 + n_1)(m_4 + n_3) > r n_1 \sigma_1$. We conclude that the origin is a stable equilibrium point if and only if~\eqref{eq:eigenvorigin3} is satisfied. 
\end{proof}

\end{document}